\documentclass[11pt]{amsart}

\usepackage{amscd,latexsym,amsthm,amsfonts,amssymb,amsmath,amsxtra,dsfont,bbm}%,IEEEtrantools}%mathdots}

\usepackage{ytableau}
\ytableausetup{smalltableaux}

\usepackage{braket}

\usepackage{MnSymbol}

\usepackage{multirow}

\usepackage[all]{xy}

\usepackage[mathscr]{eucal}
\usepackage[pdftex,bookmarks,colorlinks,pdfmenubar]{hyperref}

\usepackage{verbatim}

\renewcommand\theequation{\thesection.\arabic{equation}}

\newcommand{\BC}{{\mathbb {C}}}

\newcommand{\BG}{{\mathbb {G}}}
\newcommand{\BH}{{\mathbb {H}}}

\newcommand{\BN}{{\mathbb {N}}}

\newcommand{\BQ}{{\mathbb {Q}}}
\newcommand{\BR}{{\mathbb {R}}}

\newcommand{\BZ}{{\mathbb {Z}}}

\newcommand{\CC}{{\mathcal {C}}}

\newcommand{\CE}{{\mathcal {E}}}

\newcommand{\CH}{{\mathcal {H}}}
\newcommand{\CI}{{\mathcal {I}}}
\newcommand{\CJ}{{\mathcal {J}}}

\newcommand{\CL}{{\mathcal {L}}}

\newcommand{\CN}{{\mathcal {N}}}
\newcommand{\CO}{{\mathcal {O}}}

\newcommand{\CS}{{\mathcal {S}}}

\newcommand{\CU}{{\mathcal {U}}}

\newcommand{\CW}{{\mathcal {W}}}
\newcommand{\CX}{{\mathcal {X}}}

\newcommand{\CZ}{{\mathcal {Z}}}

\newcommand{\FB}{{\mathfrak {B}}}

\newcommand{\FD}{{\mathfrak {D}}}

\newcommand{\Fa}{{\mathfrak {a}}}

\newcommand{\Fc}{{\mathfrak {c}}}
\newcommand{\Fd}{{\mathfrak {d}}}
\newcommand{\Fe}{{\mathfrak {e}}}
\newcommand{\Ff}{{\mathfrak {f}}}
\newcommand{\Fg}{{\mathfrak {g}}}

\newcommand{\Fl}{{\mathfrak {l}}}

\newcommand{\Fn}{{\mathfrak {n}}}
\newcommand{\Fo}{{\mathfrak {o}}}
\newcommand{\Fp}{{\mathfrak {p}}}

\newcommand{\Fr}{{\mathfrak {r}}}
\newcommand{\Fs}{{\mathfrak {s}}}

\newcommand{\Fu}{{\mathfrak {u}}}

\newcommand{\RG}{{\mathrm {G}}}

\newcommand{\RI}{{\mathrm {I}}}

\newcommand{\RO}{{\mathrm {O}}}

\newcommand{\RU}{{\mathrm {U}}}

\newcommand{\ScO}{{\mathscr {O}}}

\newcommand{\Ad}{{\mathrm{Ad}}}

\newcommand{\an}{{\mathrm{an}}}

\newcommand{\bp}{{\mathrm{bp}}}

\newcommand{\disc}{{\mathrm{disc}}}

\newcommand{\End}{{\mathrm{End}}}

\newcommand{\Gal}{{\mathrm{Gal}}}
\newcommand{\GL}{{\mathrm{GL}}}

\newcommand{\gen}{{\mathrm{gen}}}
\newcommand{\gp}{{\mathrm{gp}}}

\newcommand{\Hom}{{\mathrm{Hom}}}

\newcommand{\Ind}{{\mathrm{Ind}}}
\newcommand{\ind}{{\mathrm{ind}}}

\newcommand{\Isom}{{\mathrm{Isom}}}

\newcommand{\Jac}{{\mathrm{Jac}}}

\newcommand{\Mp}{{\mathrm{Mp}}}

\newcommand{\nsd}{{\mathrm{nsd}}}

\renewcommand{\Re}{{\mathrm{Re}}}

\newcommand{\Res}{{\mathrm{Res}}}

\newcommand{\SL}{{\mathrm{SL}}}

\newcommand{\Spec}{{\mathrm{Spec}}}
\newcommand{\SO}{{\mathrm{SO}}}

\newcommand{\Sp}{{\mathrm{Sp}}}

\newcommand{\st}{{\mathrm{st}}}
\newcommand{\Span}{{\mathrm{Span}}}

\newcommand{\tr}{{\mathrm{tr}}}

\newcommand{\udl}{\underline}
\newcommand{\wt}{\widetilde}
\newcommand{\wh}{\widehat}

\newcommand{\apair}[1]{\left\langle {#1} \right\rangle}

\newcommand{\bs}{\backslash}

\def\diag{{\rm diag}}

\def\eps{{\epsilon}}

\newtheorem{thm}{Theorem}[section]
\newtheorem{cor}[thm]{Corollary}
\newtheorem{lem}[thm]{Lemma}
\newtheorem{prop}[thm]{Proposition}
\newtheorem {conj}[thm]{Conjecture}
\newtheorem {ques/conj}[thm]{Question/Conjecture}
\newtheorem {ques}[thm]{Question}
\newtheorem{defn}[thm]{Definition}

\usepackage[numbers]{natbib}

\begin{document}
\renewcommand{\theequation}{\arabic{equation}}
\numberwithin{equation}{section}

\title[Branching Law and Generic $L$-packet]{Arithmetic Branching Law and Generic $L$-Packets}

\author{Cheng Chen}
\address{School of Mathematics, University of Minnesota, Minneapolis, MN 55455, USA}
\email{chen5968@umn.edu}

\author{Dihua Jiang}
\address{School of Mathematics, University of Minnesota, Minneapolis, MN 55455, USA}
\email{dhjiang@math.umn.edu}

\author{Dongwen Liu}
\address{School of Mathematical Sciences, Zhejiang University, Hangzhou 310058, Zhejiang, China}
\email{maliu@zju.edu.cn}

\author{Lei Zhang}
\address{Department of Mathematics,
National University of Singapore,
Singapore 119076}
\email{matzhlei@nus.edu.sg}


\date{\today}

\thanks{The research of C. Chen and D. Jiang is supported in part by the NSF Grant DMS--2200890; that of D. Liu is supported in part by National Key R\&D Program of China No. 2022YFA1005300 and National Natural Science Foundation of China No. 12171421; that of L. Zhang is supported by AcRF Tier 1 grants 	A-0004274-00-00 and A-0004279-00-00 of the National University of Singapore.}

\begin{abstract}
Let $G$ be a classical group defined over a local field $F$ of characteristic zero. For any irreducible admissible representation $\pi$ of $G(F)$, which is of Casselman-Wallach type if $F$ is archimedean, we extend the study of spectral decomposition of local descents in \cite{JZ18} for special orthogonal groups over non-archimedean local fields to more general classical groups over any local field $F$. In particular, if $\pi$ has a generic local $L$-parameter, we introduce the {\sl spectral first occurrence index} $\Ff_\Fs(\pi)$ and the {\sl arithmetic first occurrence index} $\Ff_{\Fa}(\pi)$ of $\pi$ and prove in Theorem \ref{thm:SA-foi} that $\Ff_\Fs(\pi)=\Ff_{\Fa}(\pi)$. 
Based on the theory of consecutive descents of enhanced $L$-parameters developed in \cite{JLZ22}, we are able to show in Theorem \ref{thm:ADS} 
that the first descent spectrum consists of all discrete series representations, which determines explicitly the branching decomposition problem by means of the relevant arithmetic data and extends the main result (\cite[Theorem 1.7]{JZ18}) to the great generality. 
\end{abstract}

\maketitle
\tableofcontents

\section{Introduction}

%%%%%%%%%%%%%%%%%%%%%%%%%%%%%

Let $F$ be a local field of characteristic zero. Let $G_n^*$ be $F$-quasisplit classical groups and $G_n$ be an $F$-pure inner form of $G_n^*$ (Section \ref{ssec-CGIF}). 
Let $\Pi_F(G_n)$ be the set of equivalence classes of irreducible admissible representations
$\pi$ of $G_n(F)$, which are of Casselman-Wallach type if $F$ is archimedean. The local Gan-Gross-Prasad conjecture
for classical groups detects certain information on the branching law for either Bessel setting or Fourier-Jacobi setting by means of certain relations of the relevant local root numbers on the arithmetic side of the local Langlands correspondence for the classical groups over $F$ (\cite{GGP12} and \cite{GGP20}). More precisely, for an integer $p_1>0$, 
which is even if $G_n(F)$ is symplectic or metaplectic; and is odd if $G_n$ is orthogonal, 
there is a twisted Jacquet module $\CJ_{\CO_{p_1}}(\pi)$ of either Bessel type or Fourier-Jacobi type for $\pi\in\Pi_F(G_n)$ as defined 
in Section \ref{ssec-SFO}, which is a representation of $H_{\CO_{p_1}}(F)$. The $F$-quasisplit form of $H_{\CO_{p_1}}(F)$ that will be explicitly explained in \eqref{relevantpair} is
of the form:
\begin{equation} \label{rp}
(G_n^*, H_{\lfloor(\frak{n}-p_1)/2\rfloor}^*)=
\begin{cases} 
(\SO_\frak{n}, \SO_{\frak{n}-p_1}), & \textrm{if }E=F\textrm{ and }\epsilon=1,\\
(\Sp_{2n}, \Mp_{2n-p_1}) \textrm{ or }(\Mp_{2n}, \Sp_{2n-p_1}), & \textrm{if }E=F\textrm{ and }\epsilon=-1,\\
(\RU_\frak{n}, \RU_{\frak{n}-p_1}), & \textrm{if }E=F(\delta).
\end{cases} 
\end{equation}
If $p_1=\Fn$, we regard $H_{\CO_{p_1}}^*$ as the trivial group. 
If $(V,q)$ is a non-degenerate $\epsilon$-Hermitian vector space such that $G_n:=\Isom(V,q)^\circ$, then we must have $p_1\leq \Fn=\dim V$. 
Here $\CO_{p_1}$ is an $F$-rational nilpotent orbit in $\Fg_n(F)$, with $\Fg_n$ the Lie algebra of $G_n$, which is contained in the $F$-stable nilpotent orbit $\CO_{p_1}^\st$ associated with partition $[p_1,1^{\Fn-p_1}]$. See Section \ref{ssec-SFO} for more details. 

The branching problem related to the local Gan-Gross-Prasad conjecture is to detect 
such representations $\sigma\in\Pi_F(H_{\CO_{p_1}})$ that 
\begin{equation}\label{GGP-BL}
    \Hom_{H_{\CO_{p_1}}(F)}(\CJ_{\CO_{p_1}}(\pi)\,\wh\otimes\,\sigma, \mathbbm{1})\neq 0
\end{equation}
by means of the enhanced $L$-parameters of $\pi$ and $\sigma$. Here and thereafter, $\mathbbm{1}$ denotes the trivial representation.
In this paper, we continue our previous work (\cite{JZ18} and \cite{JLZ22}) to obtain
much refined structures for the branching problem in this setting, by using the method of
descents of enhanced $L$-parameters as developed in \cite{JZ18} and \cite{JLZ22}. 

We denote by $\Spec_{\CO_{p_1}}(\pi)$ the set of the contragredient representations 
$\sigma^\vee\in \Pi_F(H_{\CO_{p_1}})$ with \eqref{GGP-BL} true for $\sigma$, and call $\Spec_{\CO_{p_1}}(\pi)$ 
the spectrum of $\pi$ along $\CO_{p_1}$ or the spectrum of the module $\CJ_{\CO_{p_1}}(\pi)$. 
We define the spectrum of $\pi$ at $p_1$ to be 
\begin{equation}\label{specp1}
    \Spec_{p_1}(\pi):=\bigcup_{\CO_{p_1}\subset\CO_{p_1}^\st}\Spec_{\CO_{p_1}}(\pi).
\end{equation}

\begin{ques}[Existence of Spectrum]\label{ques:NE}
For a given $\pi\in\Pi_F(G_n)$, is the spectrum $\Spec_{p_1}(\pi)$ nonempty for some choice of integer $p_1$ with $0<p_1\leq\Fn$?
\end{ques}

By using the Vogan version of the local Langlands correspondence (Section \ref{ssec-LVP}), 
the local Gan-Gross-Prasad conjecture (Section \ref{ssec-LGGP}) and the descent of enhanced 
$L$-parameters (Section \ref{ssec-DELP}), we are able to prove that if a $\pi\in\Pi_F(G_n)$ has a generic local $L$-parameter, then there exists at least one 
integer $p_1$ with $0<p_1\leq\Fn$, such that the spectrum $\Spec_{p_1}(\pi)$ is not empty (Proposition \ref{prop:sne}).

As a basic question in the general branching problem (see \cite{Ko19}), it is interesting to ask the following question. 
%\begin{ques}[Existence of Discrete Series]\label{ques:DS}
%If the spectrum $\Spec_{\CO_{p_1}}(\pi)$ is not empty for some $p_1$ with $0<p_1\leq\Fn$, does the spectrum $\Spec_{\CO_{p_1}}(\pi)$ contain an irreducible  discrete series representation $\sigma$ of %$H_{\CO_{p_1}}(F)$?
%\end{ques}
\begin{ques}[Discrete Series Spectrum]\label{ques:DSS}
For any $\pi\in\Pi_F(G_n)$, does there exist an integer $p_1$ with $0<p_1\leq\Fn$ such that if the spectrum $\Spec_{\CO_{p_1}}(\pi)$ is not empty for some $F$-rational nilpotent orbit $\CO_{p_1}$, then $\Spec_{\CO_{p_1}}(\pi)$ consists only of irreducible discrete series representations $\sigma^\vee$ of $H_{\CO_{p_1}}(F)$?
\end{ques}

In order to understand Questions \ref{ques:NE} and \ref{ques:DSS}, we introduce a notion of {\sl spectral first occurrence index}, which is denoted by 
$\Ff_\Fs(\pi)$ for any $\pi\in\Pi_F(G_n)$ and is defined by 
\begin{equation}
    \Ff_\Fs(\pi):=\max\{p_1\mid 0<p_1\leq\Fn, \Spec_{\CO_{p_1}}(\pi)\neq\emptyset\ 
    {\rm for\ some}\ \CO_{p_1}\subset\CO_{p_1}^\st\}.
\end{equation}
In this case, we call the spectrum $\Spec_{\Ff_\Fs}(\pi)$ the {\sl first descent spectrum} of $\pi$ if $\Ff_\Fs=\Ff_\Fs(\pi)$. One of the main results of this paper is the following theorem, which provides a positive answer to a stronger version of Question \ref{ques:DSS}. 

\begin{thm}\label{thm:DSS}
If  $\pi\in\Pi_F(G_n)$ has a generic $L$-parameter, then the first descent 
spectrum $\Spec_{\Ff_\Fs}(\pi)$ consists only of discrete series representations, 
which are representations of 
$H_{\CO_{\Ff_\Fs}}(F)$ for some $F$-raitonal nilpotent orbit $\CO_{\Ff_\Fs}\subset\CO_{\Ff_\Fs}^\st$, where $\Ff_\Fs=\Ff_\Fs(\pi)$ is the spectral first occurrence index of $\pi$.  
\end{thm}

In order to prove Theorem \ref{thm:DSS}, we use the Vogan version of the local Langlands 
correspondence to obtain a unique enhanced $L$-parameter $(\varphi,\mu)$ for each $\pi\in\Pi_F(G_n)$. Here $\varphi$ is a generic $L$-parameter of an $F$-quasisplit classical group $G_n^*$, of which $G_n$ is an $F$-pure inner form, and is $G_n$-relevant, so that 
$\pi\in\Pi_\varphi(G_n)$, the local $L$-packet associated with $\varphi$. And $\mu$ is a character of the component group $\CS_\varphi$ of $\varphi$ determined by $\pi$ and $G_n$. 
We refer to Section \ref{ssec-LVP} for the definition of $\CS_\varphi$. Finally,
the uniqueness of the correspondence between $\pi$ and $(\varphi,\mu)$ depends on a choice of the Whittaker datum that is determined by a number $a\in F^\times$ (or $a\in\CZ$ as in \eqref{CZ}, more precisely). In this way, we write
\begin{equation}
    \pi=\pi_a(\varphi,\mu)\ {\rm and}\ \iota_a(\pi)=(\varphi,\mu)
\end{equation}
which realizes a unique one-to-one correspondence between the local Vogan packet associated 
with $\varphi$, which is denoted by $\Pi_\varphi[G_n^*]$, and the character group 
(the Pontryagin dual) $\wh{\CS_\varphi}$ of the component group $\CS_\varphi$ associated with 
$\varphi$. Here the local Vogan packet $\Pi_\varphi[G_n^*]$ is defined to be
\[
\Pi_\varphi[G_n^*]:=\bigcup_{G_n}\Pi_\varphi(G_n)
\]
where $G_n$ runs over all $F$-pure inner forms of the $F$-quasisplit group $G_n^*$. The 
local $L$-packet $\Pi_\varphi(G_n)$ is defined to be empty if $\varphi$ is not $G_n$-relevant. 
The method of descents of enhanced $L$-parameters $(\varphi,\mu)$ was first studied in \cite{JZ18} and is further developed in \cite{JLZ22}. This method defines the descent of 
an enhanced $L$-parameter $(\varphi,\mu)$ and the notion of 
{\sl first occurrence index} of $(\varphi,\mu)$, which is denoted by $\Fl_0(\varphi,\mu)$. 
The descent of $(\varphi,\mu)$ at the first occurrence index $\Fl_0=\Fl_0(\varphi,\mu)$ is 
denoted by $\FD_{\Fl_0}(\varphi,\mu)$ as in Definition \ref{defn:pd}. We call $\FD_{\Fl_0}(\varphi,\mu)$ the {\sl first descent} of $(\varphi,\mu)$, which by definition 
consists of some enhanced $L$-parameters of the $F$-quasisplit $H_{\lfloor(\frak{n}-\Fl_0)/2\rfloor}^*(F)$ 
(as in \eqref{rp} with $p_1=\Fl_0$). 

Theorem \ref{thm:DFD}, which is \cite[Theorem 5.3]{JLZ22}, asserts that 
the first descent $\FD_{\Fl_0}(\varphi,\mu)$ of $(\varphi,\mu)$ consists only 
such enhanced $L$-parameters $(\phi,\nu)$ of $H_{\lfloor(\frak{n}-\Fl_0)/2\rfloor}^*(F)$ 
that $\phi$ are discrete $L$-parameters of $H_{\lfloor(\frak{n}-\Fl_0)/2\rfloor}^*(F)$. 
This motivates the following theorem, which is one of the basic properties that relate the descent on the representation side (or the spectral side) and that on the enhanced $L$-parameter side 
(or arithmetic side). 

\begin{thm}[Spectral-Arithmetic]\label{thm:SA-foi}
Given any generic $L$-parameter $\varphi$ of $G_n^*(F)$, for any 
$\pi\in\Pi_F[G_n^*]$, if 
\[
\pi=\pi_a(\varphi,\mu)
\]
holds for some $\mu\in\wh{\CS_\varphi}$, then the spectral first occurrence index $\Ff_\Fs(\pi)$ is equal to the arithmetic first occurrence index $\Ff_\Fa(\pi)$, i.e.
\[
\Ff_\Fs(\pi)=\Ff_\Fa(\pi),
\]
where $\Ff_\Fa(\pi):=\Fl_0(\varphi,\mu)$, the first occurrence index of $(\varphi,\mu)$.
\end{thm}

We are going to prove Theorem \ref{thm:SA-foi} in Section \ref{sec-PC-SAFOI}. 
Combining Theorem \ref{thm:SA-foi} with Theorem \ref{thm:AFD}, we obtain the following discreteness of the first descent spectrum, which is clearly a refinement of  Theorem~\ref{thm:DSS}. 

\begin{thm}[Discreteness of the First Descent Spectrum]\label{thm:ADS}
Given any generic $L$-parameter $\varphi$ of $G_n^*(F)$, for any $\pi\in\Pi_F[G_n^*]$, if 
\[
\pi=\pi_a(\varphi,\mu)
\]
holds for some $\mu\in\wh{\CS_\varphi}$, then the first descent spectrum $\Spec_{\Ff_\Fs}(\pi)$
consists exactly of the discrete series representations whose enhanced $L$-parameters belong to the first descent $\FD_{\Fl_0}(\varphi,\mu)$ of $(\varphi,\mu)$.
\end{thm}

As a consequence, we are able to establish the following submodule theorem, whose proof is given in Section \ref{ssec-ST}.

\begin{thm}[Submodule]\label{thm:SM}
For any given $\pi\in\Pi_F(G_n)$ with generic $L$-parameter, there exists an integer $p_1$ 
with $0<p_1\leq\Fn$, such that 
the representation $\pi$ can be embedded as an irreducible submodule into the following 
induced representation
\[
\Ind^{G_n(F)}_{R_{\CO_{p_1}}(F)}(\sigma^\vee\otimes\theta_{\CO_{p_1}}),
\]
where $\sigma$ is an irreducible discrete series representation of $H_{\CO_{p_1}}(F)$, and $(R_{\CO_{p_1}}, \theta_{\CO_{p_1}})$ is given by \eqref{eq:submodule}. 
\end{thm}

Theorem \ref{thm:ADS} is a main result of this paper and is the extension of the main result (\cite[Theorem 1.7]{JZ18}) to all classical groups (considered in the paper) over all local fields of characteristic zero. We would like to point out that the method used in this paper to prove Theorem \ref{thm:ADS} is different from that used in the proof of \cite[Theorem 1.7]{JZ18}. Just as in \cite{JZ18} and in the local Gan-Gross-Prasad conjecture, our results determine the structure of the representations occurring in the first descent spectrum $\Spec_{\Ff_\Fs}(\pi)$, without 
full information about the twisted Jacquet modules $\CJ_{\CO_{p_1}}(\pi)$, which is expected to be a difficult problem in the understanding of the branching problem for infinite-dimensional representations. 

\quad

The paper is organized as follows. Section \ref{sec-CGWF} studies the local descent theory for representation data. Section \ref{ssec-CGIF} recalls the classical groups and their pure inner forms for the setting of the local Gan-Gross-Prasad conjecture from \cite{GGP12, JZ18, JZ20, JLZ22}. Section \ref{ssec-TJM-p1} recalls from \cite{JLZ22} the definition of 
twisted Jacquet modules related to the local descent theory. Based on the non-vanishing in Proposition \ref{prop:sne}, we are able to define the notion of the {\sl spectral first occurrence index} of $\pi$. In Section \ref{ssec-SPI}, we discuss the relation between local descent and parabolic induction and formulate the most technical result of the paper in 
Proposition \ref{prop:PI}. Section \ref{sec-AD} considers the local descent theory for arithmetic data. Section \ref{ssec-LVP} recalls from \cite{JLZ22} the basic structure of enhanced 
local $L$-parameters involved in this paper and basic facts related to the local Vogan packets. Section \ref{ssec-LGGP} recalls from \cite{GGP12} the local Gan-Gross-Prasad 
conjecture for all classical groups and from \cite{JLZ22} the relevant explicit arithmetic data associated with the local Gan-Gross-Prasad conjecture. The relevant twisted distinguished characters are the most important ingredient in the local Gan-Gross-Prasad conjecture and are recalled in Section \ref{ssec-DC}. The descent of the enhanced local 
$L$-parameters that was first considered in \cite{JZ18} for special orthogonal groups over $p$-adic local fields and was developed in \cite{JLZ22} for all classical groups over all local fields of characteristic zero is recalled in Section \ref{ssec-DELP}. Based on the tower property (Proposition \ref{prop:TP}), we are able to define the notion of the 
{\sl arithmetic first index} of $\pi$. The main result in the local descent theory of arithmetic data is Theorem \ref{thm:DFD}. which was proved in \cite{JLZ22}. Section \ref{ssec-PSNE} proves Proposition \ref{prop:sne}. Section \ref{sec-SDAFD} proves the main results of the paper. Section \ref{ssec-afd} establishes the discreteness of the arithmetic first descent (Theorem \ref{thm:AFD}). Combined with Theorem \ref{thm:SA-foi}, which is proved in Section \ref{sec-PC-SAFOI}, we establish Theorem \ref{thm:ADS}. Finally, Section 
\ref{ssec-ST} proves Theorem \ref{thm:SM}.  The last section (Section \ref{sec-PPI}) is devoted to the proof of Proposition \ref{prop:PI}. In fact, we prove a stronger result, 
which is Proposition \ref{pro: multiplicity}.
According to \cite{MW12, GI16, Ch23}, it remains to prove Proposition \ref{pro: multiplicity} in the Fourier-Jacobi cases when $F$ is non-archimedean, which is done in 
Proposition \ref{prop: three ingredients}, following the method of \cite{MW12}.

\section{Classical Groups and Algebraic Wavefront Sets }\label{sec-CGWF}

%%%%%%%%%%%%%%%%%%%%%%%%%%%%%%%%%%%%%%%%%%%%%%%%%%%%%%%%%

\subsection{Classical groups and their pure inner forms}\label{ssec-CGIF}
The classical groups considered in this paper include the unitary groups  $\RU_{\Fn}$, special orthogonal groups $\SO_{\Fn}$, symplectic groups $\Sp_{2n}$ and metaplectic groups $\Mp_{2n}$,
following the notations in \cite{JLZ22} and also in \cite{JZ18} and \cite{JZ20}, which are compatible with those in \cite{GGP12}. 

Let $F$ be a local field of characteristic zero. As in \cite{JLZ22}, we assume that $F\neq \BC$. 
Thus if $F$ is archimedean, we mean that $F$ is real; otherwise, we mean that $F$ is $p$-adic, which is a finite extension of the field $\BQ_p$ of $p$-adic numbers for some prime $p$.
Let $F(\delta)$ be a quadratic field extension of $F$, with $\delta=\sqrt{d}$ for a non-square $d\in F^\times$.
Let $E$ be either $F$ or $F(\delta)$.
Denote by $\Fc\colon x\mapsto \bar{x}$ \label{pg:iota} the unique nontrivial element in the Galois group $\Gal(E/F)$ if $E\ne F$; and $\Fc={\rm id}_E$, $\bar{x}=x$ if $E=F$.

Let $(V,q_V)$ be an $\Fn$-dimensional vector space $V$ over $E$, equipped with a  non-degenerate $\epsilon$-Hermitian form
$q=q_V$, where $\epsilon=\pm 1$.
More precisely, if $E=F$, $q$ is symmetric or symplectic; and
if $E=F(\delta)\neq F$, $q$ is Hermitian or skew-Hermitian.
Write $\eps_q=-1$ if $q$ is symplectic or skew-Hermitian; otherwise $\eps_q=1$.
Denote by $G_n=\Isom(V,q)^\circ$ (or $\Mp_{2n}(F)$ in the metaplectic case) the identity connected component of
the isometry group of the space $(V,q)$, with $n=\lfloor\frac{\Fn}{2}\rfloor$.
Let $\Fr$ be the Witt index of $(V,q)$
and $(V_{\an},q)$ be the $F$-anisotropic kernel of $(V,q)$ with dimension $\Fd_0$.
Then the $F$-rank of $G_n$ is $\Fr$ and $\Fd_0=\Fn-2\Fr$. As in \cite[Section 2]{JLZ22}, 
the following 
$$
V=V^{+}\oplus V_{\an}\oplus V^{-}
$$
is a polar decomposition for $V$, where $V^{\pm}$ are maximal totally isotropic subspaces and dual to each other.
Take a basis $\{e_{\pm 1},\dots,e_{\pm \Fr}\}$ for $V^{\pm}$ such that
$$
q(e_{i},e_{-j})=\delta_{i,j}
$$
for all $1\leq i,j\leq\Fr$.
Choose an orthogonal basis $\{e'_{1},\dots,e'_{\Fd_0}\}$ of $V_{\an}$ and put
$$
d_i=q(e'_{i},e'_{i})\in F^\times \text{ or }\delta\cdot F^\times, \text{ for }1\leq i\leq \Fd_0.
$$
Note that if $(V,q)$ is a Hermitian space, then
$(V,\delta \cdot q)$ is skew-Hermitian and $\Isom(V,q)=\Isom(v,\delta\cdot q)$ gives us the identical unitary groups.
However, for the induction purpose, we will consider both cases as in \cite{JLZ22}.
If $(V,q)$ is symplectic, then $\Fn=2n=2\Fr$ and $V_{\an}=\{0\}$.
We put the above bases together in the following order to form a basis of $(V,q)$
\begin{equation}\label{eq:basis}
\FB\colon\quad e_{1},\dots,e_{\Fr},e'_{1},\dots,e'_{\Fd_0},e_{-\Fr},\dots,e_{-1},
\end{equation}
and fix the following full isotropic flag in $(V,q)$:
$$
\Span\{e_{1}\}\subset\Span\{e_{1},e_{2}\}\subset
\cdots\subset
\Span\{e_{1},\dots,e_{\Fr}\},
$$
which defines a minimal parabolic $F$-subgroup $P_0$ of $G_n$.
With respect to the order of the basis in \eqref{eq:basis}, the group $G_n$ is also defined with respect to the following matrix:
\begin{equation} \label{eq:J}
J_{\Fr}^\Fn=\begin{pmatrix}
&&w_\Fr\\&J_{0}^{\Fd_0}&\\ \epsilon_q w_{\Fr}&&
\end{pmatrix}_{\Fn\times\Fn}
\text{ where }
w_\Fr=\begin{pmatrix}
&w_{\Fr-1}\\1&	
\end{pmatrix}_{\Fr\times \Fr},
\end{equation}
which is defined inductively, and
$J_{0}^{\Fd_0}=\diag\{d_{1},\dots,d_{\Fd_0}\}$.
Moreover, the Lie algebra $\Fg_n$ of $G_n$ is defined by
\[
\Fg_n=\{A\in {\rm {End}}_E(V) \mid \bar{A}^{t}J_{\Fr}^\Fn+ J_{\Fr}^\Fn A=0\},	
\]
with the Lie bracket $[A,B]=AB-BA$.
For simplicity, set $A^*=(J_{\Fr}^\Fn)^{-1} \bar{A}^{t} J_{\Fr}^\Fn$.
As in \cite[Section 2.1]{JLZ22}, we take the following $\Ad(G_n)$-invariant non-degenerate $F$-bilinear form $\kappa$ on $\Fg_n\times \Fg_n$:
\[
\kappa(A,B)=\tr(AB^*)/2=\tr(A^*B)/2.	
\]
Alternatively, $\kappa$ can be computed by 
\begin{equation}\label{eq:kappa}
2\kappa(A,B)= \sum_{i=1}^{\Fr}\apair{A(e_{i}),B(e_{-i})}+\epsilon_q \apair{A(e_{-i}),B(e_{i})}+\sum_{j=1}^{\Fd_0}\frac{\apair{A(e'_{j}),B(e'_{j})}}{\apair{e'_{j},e'_{j}}}.	
\end{equation}
Note that this $\Ad(G_n)$-invariant symmetric bilinear form is the same as the one in \cite{GZ14} and is proportional to the Killing form.
For an $\eps$-Hermitian space $V$ of dimension $\Fn$, we define its discriminant by
\begin{equation}
\disc(V)=(-1)^{\frac{\Fn(\Fn-1)}{2}}\det(V) \in   \left\{ \begin{array}{ll}  \delta^\Fn \cdot F^\times /\BN E^\times,	& \textrm{if }E=F(\delta)\textrm{ and }\epsilon = -1, \\
F^\times / \BN E^\times, & \textrm{otherwise,}
\end{array}\right.
\end{equation}
where $\BN E^\times:=\{x\bar{x} \mid x\in E^\times\}$, unless $E=F$
 and $V$ is a symplectic space in which case this definition is not applicable.

\subsection{Twisted Jacquet modules associated with $[p_1,1^{\Fn-p_1}]$}\label{ssec-TJM-p1}

We review here the discussion in \cite[Section 4.2]{JLZ22}, which specializes to the partition of type $[p_1,1^{\Fn-p_1}]$. 

Denote by $\CO^{\st}_{p_1}$ the unique $F$-stable nilpotent orbit associated with the partition $[p_1,1^{\Fn-p_1}]$.
To parameterize the $F$-rational nilpotent orbits in $\CO^{\st}_{p_1}$,
we need to assign  an $\eps_1$-Hermitian form $(V^h_{(p_1)},q^h_{(p_1)})$ and an $\eps$-Hermitian form $(V_{(1)},q_{(1)})$
where $\eps_1=(-1)^{p_1-1}\eps$ and $\dim V^h_{(p_1)}=1$.
By definition,  $V_{(p_1)}$ has the polar decomposition $V_{(p_1)}^+ \oplus \Fe\oplus V_{(p_1)}^-$,
where $\Fe$ is an anisotropic line if $p_1$ is odd, and is zero otherwise, and
$V_{(p_1)}^+$ and $V_{(p_1)}^-$ are dual to each other.
In particular when $p_1$ is odd,  $q_{(p_1)}$ can be determined by $q_{(p_1)}\vert_{\Fe}=q^h_{(p_1)}$.
The constructed Young tableaux for $(V,q)$ with the partition $[p_1,1^{\Fn-p_1}]$ is admissible if and only if
\begin{equation}\label{eq:admissible-Young-p1}
(V,q)\cong (V_{(p_1)}^+\oplus \Fe \oplus V_{(p_1)}^-,q_{(p_1)})\oplus (V_{(1)},q_{(1)}). 	
\end{equation}
We refer to \cite[Section 3.1]{GZ14} and \cite[Section 2.2]{JLZ22} for the detailed discussion in general situation. 

From \cite[Section 4.2]{JLZ22},  in order for  $\CO^{\st}_{p_1}$ to be defined over $F$ (i.e., $\CO^{\st}_{p_1}\cap \CN_F(\Fg_n)\ne \varnothing$), the following conditions must be satisfied:
\begin{itemize}
    \item If $V$ is orthogonal and $\Fn\ne 2\Fr$, then $p_1$ is odd, and
$\CO^{\st}_{p_1}$ is defined over $F$ if and only if $p_1\leq 2\Fr+1$.
\item If $V$ is orthogonal and $\Fn=2\Fr$, then $p_1$ is odd, and
$\CO^{\st}_{p_1}$ is defined over $F$ if and only if $p_1\leq 2\Fr-1$.
\item If $V$ is symplectic, then $p_1$ is even, and $\CO^{\st}_{p_1}$ is defined over $F$ for all $p_1\leq \Fn$.
\item If $G_n$ is unitary, then there is no constrain  on $p_1$, and $\CO^{\st}_{p_1}$ is defined over $F$ if and only if $p_1\leq 2\Fr+1$.
\end{itemize}
From \cite[Section 4.2]{JLZ22}, there is an explicit choice of ${\frak {sl}}_2$-triple $\{X,\hbar, Y\}$ for each $F$-rational orbit in $\CO^{\st}_{p_1}$ with $X\in \CO^{\st}_{p_1}$.
With the basis given in \eqref{eq:basis}, the following is a choice of $\hbar\in \Fg_n$, for $x\in \FB$,
\[
\hbar(x)= \begin{cases}
	\mp(p_1-2i+1)e_{\pm i}  &\text{ if $x=e_{\pm i}$ and } 1\leq i\leq  \lfloor\frac{p_1}{2}\rfloor,\\
	0 &\text{ otherwise,}
\end{cases}
\]
which is uniform for all $F$-rational orbits. 
With such a chosen $\hbar$, we obtain the following eigenspaces 
\begin{align}\label{eigensp}
    V_j=\{v\in V\ \mid \hbar(v)=jv\}
\end{align}
for $j\in\BZ$. 
For $v,w\in V$, define the element $A_{v,w}$ in $\Fg_n\subset\End(V)$ by
\begin{equation}
A_{v,w}(x)= \eps q(x,v)w- q(x,w)v.		
\end{equation}
Throughout this paper, we choose and fix a vector $e$ in $V$ based on the parity of $p_1$:
\begin{itemize}
	\item if $p_1$ is odd, we take $e$ to be an anisotropic vector in $V_0$;
	\item if $p_1$ is even, we take $e=\frac{\eps}{2}e_{-m}$,
\end{itemize}
where 
%$V_{0}=\{v\in V\colon \hbar(v)=0\}$ and  
$m=\lfloor\frac{p_1}{2}\rfloor$.
Note that if $p_1$ is odd, then $V$ is not symplectic and there always exists anisotropic vectors if $V_{0}$ is not zero.
With the fixed vector $e$, \cite[(4.4)]{JLZ22} defines 
\begin{equation} \label{xe,var}
X_{e,\varsigma}=A_{ e, \bar{\varsigma} e_{-m}}+\sum_{i=1}^{m-1}A_{ e_{i+1}, \bar{\varsigma} e_{-i}}.	
\end{equation}
for $\varsigma\in E^\times$, which implies that $X_{e,\varsigma}(e_m)= \bar{\varsigma} e_{-m}$. 
After completing $X_{e,\varsigma}$ and $\hbar$ to an ${\frak {sl}}_2$-triple in $\Fg_n$, we have 
that 
\[V_{(p_1)}= {\rm {Span}}\{e_1,e_2,\dots,e_{m}\}\oplus Ee \oplus {\rm {Span}}\{e_{-m},e_{-(m-1)},\dots,e_{-1}\}\]
and
$V_{(1)}=V_{(1)}^h=V_{0}\cap e^\perp$.
Note that $V_{(p_1)}^h=Ee_{-1}$ and
\begin{equation}\label{eq:X-highest-weight}
X^{p_1-1}(e_{1})=\begin{cases}
	(-\varsigma\bar{\varsigma})^{m-1}\bar{\varsigma} e_{-1} &\text{ if $p_1$ is even,}\\
	\eps\cdot(-\varsigma\bar{\varsigma})^m\apair{e,e}e_{-1} &\text{ if $p_1$ is odd.}\\	
\end{cases}	
\end{equation}
From \eqref{eq:X-highest-weight}, we deduce that $q^{h}_{(p_1)}(e_{-1},e_{-1})$ is equivalent to the one-dimensional $(-1)^{p_1-1}\eps$-Hermitian form defined by
\begin{equation} \label{q'-form}
q'_{e,\varsigma}(e_{-1},e_{-1}):= \begin{cases}
	(-1)^{m-1}\varsigma &\text{ if $p_1$ is even,}\\
	(-1)^m \apair{e,e} &\text{ if $p_1$ is odd.}\\	
\end{cases}
\end{equation}
Denote the resulting  $\epsilon$-Hermitian form on $V_{(p_1)}\simeq Ee_{-1}\otimes F^{p_1}$ by
\begin{equation}\label{x-form}
q_{(p_1), e, \varsigma}:=q'_{e,\varsigma}\otimes q_{p_1}.
\end{equation}
Therefore, we obtain the admissible Young tableau associated to $X_{e,\varsigma}$:
\begin{equation}\label{eq:X-p-1}
([p_1,1^{\Fn-p_1}],\{(Ee_{-1},q'_{e,\varsigma}),(V_{0}\cap e^{\perp},q_V\vert_{V_{0}\cap e^{\perp}})\}).	
\end{equation}
We refer to \cite[Section 2.2]{JLZ22} for the details in the general situation. 

The twisted Jacquet module associated to the partition of type $[p_1,1^{\Fn-p_1}]$ splits into two types: the Bessel type and the Fourier-Jacobi type, according to $p_1$ is odd and even,
respectively. Note that when $p_1=1$ the corresponding stable nilpotent orbit is the zero nilpotent orbit.
However, the twisted Jacquet modules of the Bessel type and Fourier-Jacobi type for the classical groups can still be defined for $p_1=1$ (\cite{GGP12}).

For $X_{e,\varsigma}$ as in \eqref{xe,var}, following \cite[(4.8)]{JLZ22}, we define 
\begin{equation} \label{psix}
\psi_{X_{e,\varsigma},2}(\exp (u))=\psi_0(\kappa(u, X_{e,\varsigma})),\quad u\in \Fu_{X_{e,\varsigma}, 2}:=\bigoplus_{i\leq -2}\Fg_{n, i}^{\hbar},
\end{equation}
where $\Fg_{n, i}^{\hbar}$ denotes the $i$-eigenspace of ${\rm ad}(\hbar)$ acting on $\Fg_n$.
Following \cite[Lemma 4.2]{JLZ22} for the explicit computation of $\kappa(u, X_{e,\varsigma})$, 
we obtain that
\begin{equation}\label{Mx}
M_{X_{e,\varsigma}}=\Isom(V_{(1)}, q),
\end{equation}
where $V_{(1)}=V_0\cap e^\perp$. Define
\[
\Fu_{X_{e,\varsigma}}:=\bigoplus_{i\leq -1}\Fg_{n,i}^{\hbar},\quad U_{X_{e,\varsigma}}:=\exp(\Fu_{X_{e,\varsigma}}),
\]
and we specify a representation $\omega_{\psi_{X_{e,\varsigma}}}$ of $U_{X_{e,\varsigma}}$ below. 

The {\sl twisted Jacquet modules of Bessel type} are given as follows. 
If $\epsilon=1$ and $p_1$ is odd, then $\Fu_{X_{e,\varsigma}}=\Fu_{X_{e,\varsigma},2}$ and we let $\omega_{\psi_{X_{e,\varsigma}}}=\psi_{X_{e,\varsigma},2}$. For a smooth representation $(\pi, V_\pi)$ of $G_n(F)$, we define its twisted Jacquet module with respect to $(U_{X_{e,\varsigma}}, \psi_{X_{e, \varsigma}})$
\begin{equation} \label{BF-coe}
\CJ_{X_{e, \varsigma}}(\pi)=V_\pi /\overline{\Span\{\pi(u)v-\psi_{X_{e, \varsigma},2}(u)v \mid  u\in U_{X_{e, \varsigma}}, v\in V_\pi\}},
\end{equation}
as a smooth module of $M^\circ_{X_{e,\varsigma}}$.

The {\sl twisted Jacquet modules of Fourier-Jacobi type} are given as follows. 
If $\epsilon=-1$ and $p_1$ is even, then we have the Heisenberg-oscillator representation, denoted by
$(\omega_{\psi_{X_{e,\varsigma}}}, V_{\psi_{X_{e, \varsigma}}})$, of $\wt{M}_{X_{e,\varsigma}}\ltimes U_{X_{e,\varsigma}}$. For a smooth representation
$(\pi, V_\pi)$ of $G_n(F)$ (or $\wt{G}_n(F)$), we define the $U_{X_{e, \varsigma}}$-coinvariants
\begin{equation} \label{FJ-coe}
\CJ_{X_{e, \varsigma}}(\pi)=V_\pi\,\wh{\otimes}\, V^\vee_{\psi_{X_{e,\varsigma}}}/\overline{\Span\{\pi\otimes \omega^\vee_{\psi_{X_{e,\varsigma}}}(u)v-v \mid  u\in U_{X_{e, \varsigma}}, v\in V_\pi
\,\wh{\otimes}\, V^\vee_{\psi_{X_{e,\varsigma}}}\}},
\end{equation}
as a smooth module of $\wt{M}_{X_{e,\varsigma}}$ (or $M_{X_{e,\varsigma}}$).

In the skew-Hermitian case, choose a splitting character $\xi: E^\times/\BN E^\times \to \BC^\times$ whose restriction to $F^\times$ is the quadratic character $\omega_{E/F}$. Then we may and do remove the double cover and obtain the Heisenberg-oscillator representation $(\omega_{\psi_{X_{e,\varsigma}},\xi},V_{\psi_{X_{e,\varsigma}},\xi})$ of the unitary group $M_{X_{e,\varsigma}}$ in the above.

\subsection{Spectral first occurrence}\label{ssec-SFO}

From the construction of the twisted Jacquet modules associated with the partition $[p_1,1^{\Fn-p_1}]$, the equivalence class of the module $\CJ_{X_{e, \varsigma}}(\pi)$ depends 
only on the $F$-rational nilpotent orbit $\CO_{e, \varsigma}\subset\CO_{p_1}^\st$. We may write 
$\CO_{p_1}=\CO_{e, \varsigma}$ for some pair $(e, \varsigma)$, and write 
\begin{equation}\label{LD-module}
    \CJ_{\CO_{p_1}}(\pi)=\CJ_{X_{e, \varsigma}}(\pi).
\end{equation}
To uniform the notations, we write $H_{\CO_{p_1}}$ for either $M_{X_{e,\varsigma}}$ or $\wt{M}_{X_{e,\varsigma}}$, whose $F$-quasisplit forms are given as follows:
\begin{equation} \label{relevantpair}
(G_n^*, H_{\lfloor(\frak{n}-p_1)/2\rfloor}^*)=
\begin{cases} 
(\SO_\frak{n}, \SO_{\frak{n}-p_1}), & \textrm{if }E=F\textrm{ and }\epsilon=1,\\
(\Sp_{2n}, \Mp_{2n-p_1}) \textrm{ or }(\Mp_{2n}, \Sp_{2n-p_1}), & \textrm{if }E=F\textrm{ and }\epsilon=-1,\\
(\RU_\frak{n}, \RU_{\frak{n}-p_1}), & \textrm{if }E=F(\delta).
\end{cases} 
\end{equation}
If $p_1=\Fn$, we regard $H_{\CO_{p_1}}^*$ as the trivial group. We also write
\begin{equation} \label{eq:submodule}
(R_{\CO_{p_1}}, \theta_{\CO_{p_1}}):= (H_{\CO_{p_1}}\ltimes U_{X_{e,\varsigma}}, \omega_{\psi_{X_{e,\varsigma}}}).
\end{equation}

For any $\pi\in\Pi_F(G_n)$, we consider the spectral structure of the module $\CJ_{\CO_{p_1}}(\pi)$ as representation of $H_{\CO_{p_1}}(F)$.
For any $\sigma\in\Pi_F(H_{\CO_{p_1}})$, we say that $\sigma^\vee$ {\sl occurs} in 
$\CJ_{\CO_{p_1}}(\pi)$ if 
\begin{equation}\label{spec}
    \Hom_{H_{\CO_{p_1}}(F)}(\CJ_{\CO_{p_1}}(\pi)\,\wh\otimes\,\sigma, \mathbbm{1})\neq 0. 
\end{equation}
Denote by $\Spec_{\CO_{p_1}}(\pi)$ the {\sl spectrum} of the module $\CJ_{\CO_{p_1}}(\pi)$, which is a set consisting of all $\sigma^\vee\in\Pi_F(H_{\CO_{p_1}})$ that occur in $\CJ_{\CO_{p_1}}(\pi)$. 
%In \cite{JLZ22}, we define the local descents of $\pi$ using dual representations of $\sigma$. However, this will not affect the definition of spectrum descents introduced in this paper.

\begin{prop}\label{prop:sne}
If a representation $\pi\in\Pi_F(G_n)$ has a generic $L$-parameter, then there exists at least one integer $p_1$ with $0<p_1\leq\Fn$, such that the set $\Spec_{p_1}(\pi)$ that is 
the spectrum of $\pi$ at $p_1$ as defined in \eqref{specp1} is not empty. 
\end{prop}

The proof needs preparation of arithmetic data associated with $\pi\in\Pi_F(G_n)$ and will be given in Section \ref{ssec-PSNE}.

Assume that for a $\pi\in\Pi_F(G_n)$, there exists an integer $p_1$ with $0<p_1\leq \Fn$ such that the spectrum $\Spec_{\CO_{p_1}}(\pi)$ is non-empty. We define the {\sl spectral first occurrence index} $\Ff_\Fs=\Ff_\Fs(\pi)$ for such a $\pi$ to be 
\begin{equation}
    \Ff_\Fs=\Ff_\Fs(\pi)=\max\{p_1\mid 0<p_1\leq\Fn,\ \Spec_{\CO_{p_1}}(\pi)\neq \emptyset\}. 
\end{equation}
If the spectral first occurrence index $\Ff_\Fs=\Ff_\Fs(\pi)$ exists, we call the spectrum 
$\Spec_{\CO_{\Ff_\Fs}}(\pi)$ the {\sl spectral first descent} of $\pi$ (along the orbit $\CO_{\Ff_\Fs}$). 
We are going to understand the structure of the spectral first descent $\Spec_{\CO_{\Ff_\Fs}}(\pi)$ of $\pi$ by the descent of enhanced $L$-parameters as developed in 
\cite{JZ18} and \cite{JLZ22}.

\subsection{Spectrum and parabolic induction}\label{ssec-SPI}

Let $p_1$ be an integer with $0<p_1\leq\Fn$. We take $\sigma\in\Pi_F(H_{\lfloor(\frak{n}-p_1)/2\rfloor})$, 
where $H_{\lfloor(\frak{n}-p_1)/2\rfloor}$ is an $F$-pure inner form of 
$H_{\lfloor(\frak{n}-p_1)/2\rfloor}^*$.

When $H_{\lfloor(\frak{n}-p_1)/2\rfloor}$ is not metaplectic, write $\sigma$ as the Langlands quotient 
$\CL(\udl{s},\tau_1,\dots,\tau_t,\sigma_0)$
of the following standard module of $H_{\lfloor(\frak{n}-p_1)/2\rfloor}(F)$
\begin{equation}\label{SM1}
    \RI(\udl{s},\tau_1,\dots\tau_t,\sigma_0):=
    \Ind^{H_{\lfloor(\frak{n}-p_1)/2\rfloor}(F)}_{P(F)}
    (|\det|^{s_1}\tau_1\otimes\cdots\otimes|\det|^{s_t}\tau_t\otimes\sigma_0)
\end{equation}
where the Levi part of the standard parabolic subgroup $P$ is isomorphic to 
\[
\GL_{n_1}\times\cdots\times\GL_{n_t}\times H_{\lfloor(\frak{n}-p_1)/2\rfloor -p_0}
\]
with $p_0=n_1+\cdots+n_t$, and $\tau_j$ is an irreducible tempered representation of 
$\GL_{n_j}(F)$ for each $j=1,\dots,t$, 
$\sigma_0$ is an irreducible tempered representation of 
$H_{\lfloor(\frak{n}-p_1)/2\rfloor-p_0}(F)$, and $s_j\in\BR$ with 
$s_1 \geq \dots \geq s_t \geq 0$.  When $H_{\lfloor(\frak{n}-p_1)/2\rfloor}$ is metaplectic, from \cite{AB98}\cite{GS12}, 
we can write $\sigma$ as the  unique quotient $\CL(\udl{s},\tau_1,\dots,\tau_t,\sigma_0)$ of the following standard module:
\begin{equation}\label{SM-Mp}
\RI(\udl{s},\tau_1,\dots\tau_t,\sigma_0):=
    \Ind^{H_{\lfloor(\frak{n}-p_1)/2\rfloor}(F)}_{P(F)}
(|\det|^{s_1}\wt{\tau}_1\otimes\cdots\otimes|\det|^{s_t}\wt{\tau}_t\otimes\sigma_0)
\end{equation}
where the Levi part of the standard parabolic subgroup $P$ is isomorphic to 
\[
\wt{\GL}_{n_1}\times_{\pm 1}\cdots\times_{\pm1}\wt{\GL}_{n_t}\times_{\pm1} H_{\lfloor(\frak{n}-p_1)/2\rfloor -p_0}
\]
with $p_0=n_1+\cdots+n_t$, and $\wt{\tau}_j$ is the genuine representation representation of 
$\wt{\GL}_{n_j}(F)$ associated to an irreducible tempered representation $\tau_j$ of $\GL_{n_j}(F)$ with respect to the fixed character $\psi_F$ of $F$ for each $j=1,\dots,t$, 
$\sigma_0$ is an irreducible genuine tempered representation of 
$H_{\lfloor(\frak{n}-p_1)/2\rfloor-p_0}(F)$, and $s_j\in\BR$ with 
$s_1 \geq \dots \geq s_t \geq 0$.
\begin{prop}\label{prop:PI}
Assume that a representation $\pi\in\Pi_F(G_n)$ has a generic $L$-parameter. Let $p_1$ be an integer with $0<p_1\leq\Fn$. If the standard module 
\[
\RI(\udl{s},\tau_1,\dots\tau_t,\sigma_0):=
    \Ind^{H_{\lfloor(\frak{n}-p_1)/2\rfloor}(F)}_{P(F)}
    (|\det|^{s_1}\tau_1\otimes\cdots\otimes|\det|^{s_t}\tau_t\otimes\sigma_0)
\]
as given in \eqref{SM1} has the property that 
\[
\Hom_{H_{\CO_{p_1}}(F)}(\CJ_{\CO_{p_1}}(\pi)\,\wh\otimes\,\RI(\udl{s},\tau_1,\dots\tau_t,\sigma_0)
, \mathbbm{1})\neq 0
\]
for the $F$-rational nilpotent orbit $\CO_{p_1}$, then 
$\sigma_0^\vee \in\Spec_{\CO_{p_1+2p_0}}(\pi)$ for some $F$-rational nilpotent orbit 
$\CO_{p_1+2p_0}$. 
\end{prop}

If the standard module $\RI(\udl{s},\tau_1,\dots\tau_t,\sigma_0)$ is irreducible, the result
of this proposition is a consequence of the local Gan-Gross-Prasad conjecture. 
The point is to prove this proposition when the standard module
$\RI(\udl{s},\tau_1,\dots\tau_t,\sigma_0)$ is reducible. In this case, we have to show that if
\[
\Hom_{H_{\CO_{p_1}}(F)}(\CJ_{\CO_{p_1}}(\pi)\,\wh\otimes\,\RI(\udl{s},\tau_1,\dots\tau_t,\sigma_0)
, \mathbbm{1})\neq 0
\]
for the $F$-rational nilpotent orbit $\CO_{p_1}$, then we must have that 
\[
    \Hom_{H_{\CO_{p_1+2p_0}}(F)}(\CJ_{\CO_{p_1}}(\pi)\,\wh\otimes\,\sigma_0,\mathbbm{1})\neq 0
\]
for some $F$-rational nilpotent orbit $\CO_{p_1+2p_0}$.
A complete proof of Proposition \ref{prop:PI} will be given in Section \ref{sec-PPI}.

%%%%%%%%%%%%%%%%%%%%%%%%%%%%%%%%%%%%%

\section{Arithmetic Data}\label{sec-AD}

%%%%%%%%%%%%%%%%%%%%%%%%%%%%%%%%%%

\subsection{Enhanced $L$-parameters and local Vogan packets}\label{ssec-LVP}

We recall from \cite[Section 3.1]{JLZ22} the definition of enhanced $L$-parameters and their associated local Vogan packets. 

Let $\CW_E$ be the local Weil group of $E$.
Then $\CW_\BC=\BC^\times$, $\CW_\BR=\BC^\times\cup j\BC^\times $,
where $j^2=-1$ and $j z j^{-1}=\bar{z}$ for $z\in \BC^\times$; and
when $E$ is $p$-adic, $\CW_E=\CI_E\rtimes \apair{\rm {Frob}}$ is the semi-direct product of the inertia group $\CI_E$ of $E$ and a geometric Frobenius element  ${\rm {Frob}}$.
The local Langlands group (also called the Weil-Deligne group) of $E$ is defined by
\begin{align}\label{LLG}
\CL_E=\begin{cases}
	\CW_E &\text{if $E$ is archimedean},\\
	\CW_E\times\SL_2(\BC) &\text{if $E$ is $p$-adic}.	
\end{cases}	
\end{align}
We take the Galois version of the $L$-group ${^LG}:=\wh G \rtimes \Gal(E/F)$ of $G$, where $\wh G$ is the complex dual group of $G$.
Note that if $G=\Mp_{2n}$, then $\wh G=\Sp_{2n}(\BC)$, and if $G$ is $F$-split, the action of $\Gal(E/F)$ on $\wh G$ is trivial.

When $G_n=G_n^*$ is $F$-quasisplit,
an $L$-parameter $\varphi\colon\CL_F\to{^LG_n}$ of $G^*_n$ is attached to a datum $(L^*,\varphi^{L^*},\udl{\beta})$ with the following properties:
\begin{enumerate}
\item $L^*$ is a Levi subgroup of $G_n$ of the form
$$
L^*=\GL_{n_1}(E)\times\cdots\times\GL_{n_t}(E)\times G_{n_0};
$$
\item $\varphi^{L^*}$ is a local $L$-parameter of $L^*$ given by
$$
\varphi^{L^*}:=\varphi_1\oplus\cdots\oplus\varphi_t\oplus\varphi_0\ :\ \CL_{F}\rightarrow {^L\!L^*},
$$
where $\varphi_j\colon\CL_E\to\GL_{n_j}(\BC)$ is a local tempered $L$-parameter of $\GL_{n_j}(E)$ for $j=1,2,\cdots,t$, and $\varphi_0$ is a local tempered $L$-parameter
of $G_{n_0}^*$;
\item $\udl{\beta}:=(\beta_1,\cdots,\beta_t)\in \BR^t$, such that $\beta_1>\beta_2>\cdots>\beta_t>0$; and
\item the $L$-parameter $\varphi$ can be expressed as
$$
\varphi=(\varphi_1\otimes|\cdot|^{\beta_1}\oplus{}^c  \varphi_1^\vee \otimes|\cdot|^{-\beta_1})\oplus\cdots\oplus
(\varphi_t\otimes|\cdot|^{\beta_t}\oplus {}^c  \varphi_t^\vee \otimes|\cdot|^{-\beta_t})\oplus\varphi_0,
$$
where $|\cdot|$ is the absolute value of $E$ and ${}^c\varphi_j^\vee$ denotes the conjugate  dual of $\varphi_j$ for $j=1,2,\cdots,t$.
\end{enumerate}
Following \cite{Ar13}, \cite{GS12}, \cite{MW12}, \cite{KMSW14}, \cite{Mok15}, and \cite{AG17a}, with a choice of a Whittaker datum, one can form a local $L$-packet ${\Pi}_\varphi(G_n)$
for each $L$-parameter $\varphi$, which is a finite subset of $\Pi_F(G_n)$ satisfies certain
constraints in the local Langlands conjecture for $G_n$ over $F$. A local $L$-parameter $\varphi$ is called {\sl generic} if the local $L$-packet $\Pi_\varphi(G_n)$ contains a generic member, i.e. a member with 
a non-zero Whittaker model with respect to the given Whittaker datum. 

One may define ${\Pi}_\varphi[G_n^*]$ to be the local Vogan packet associated to each $L$-parameter $\varphi$, as in \cite[Sections 9--11]{GGP12}, for the classical groups under consideration:
\begin{equation}\label{lvp}
{\Pi}_\varphi[G_n^*]=\bigcup_{G_n}{\Pi}_\varphi(G_n)
\end{equation}
where $G_n$ runs over all pure inner $F$-forms of the given $F$-quasisplit $G_n^*$ and the local $L$-packet ${\Pi}_\varphi(G_n)$ is the subset of representations of $G_n$.
The $L$-packet ${\Pi}_\varphi(G_n)$ is defined to be
{\it empty} if the parameter $\varphi$ is not $G_n$-relevant. In the case that $G^*_n=\SO_{2n}$, we use the weak local Langlands correspondence by  \cite{AG17a} in the above and replace $\Pi_\varphi(G_n)$ by the equivalence classes $\Pi_\varphi(G_n)/\sim_{\rm c}$ under the outer action of $\RO_{2n}$, i.e. conjugation by an element ${\rm c}\in \RO_{2n}\setminus \SO_{2n}$ with $\det(\rm{c})=-1$.
The set of all generic local $L$-parameters of $G_n^*$, up to equivalence, is denoted by ${\Phi}_\gen(G_n^*)$.
For an $L$-parameter $\varphi$,
denote $S_\varphi$ to be the centralizer of the image of $\varphi$ in $\wh G$, and $S^\circ_\varphi$ to be its identity connected component group.
Then the component group $\CS_\varphi:=S_\varphi/S_\varphi^\circ$ is an abelian $2$-group for the classical groups under consideration. For $\varphi\in{\Phi}_\gen(G_n^*)$, the group $\CS_\varphi$ has been computed explicitly in \cite[Section 8]{GGP12} and also in \cite[Section 3.2]{JLZ22}, which is recalled briefly as follows. 
Write
\begin{equation}\label{decomp}
\varphi=\oplus_{i\in \RI}m_i\varphi_i,
\end{equation}
which is the decomposition of $\varphi$ into simple and generic ones.
The simple, generic local $L$-parameter $\varphi_i$ can be written as $\rho_i\otimes\mu_{b_i}$,  where $\rho_i$ is an $a_i$-dimensional
irreducible representation of $\CW_F$ and $\mu_{b_i}$ is the irreducible representation of $\SL_2(\BC)$ of dimension $b_i$.

In the decomposition \eqref{decomp},
$\varphi_i$ is called {\it of good parity} if $\varphi_i$ is of the same type of $\varphi$. 
We denote by $\RI_{\gp}$ the subset of
$\RI$ consisting of indices $i$ such that $\varphi_i$ is of good parity; and by $\RI_{\bp}$ the subset of
$\RI$ consisting of indices $i$ such that $\varphi_i$ is (conjugate) self-dual, but not of good parity. We set $\RI_{\nsd}:=\RI-(\RI_{\gp}\cup\RI_{\bp})$ for
the indexes of $\varphi_i$'s which are not (conjugate) self-dual.
Hence we may write $\varphi\in{\Phi}_\gen(G_n^*)$ in the following more explicit way:
\begin{equation}\label{edec}
\varphi=(\oplus_{i\in \RI_\gp}m_i\varphi_i)\oplus(\oplus_{j\in \RI_\bp}2m'_j\varphi_j)\oplus(\oplus_{k\in \RI_\nsd}m_k(\varphi_k\oplus{}^c \varphi_k^\vee)),
\end{equation}
where $2m'_j=m_j$ in \eqref{decomp} for $j\in \RI_{\bp}$.
According to this explicit decomposition, it is easy to know that
\begin{equation}\label{2grp}
\CS_\varphi\cong\BZ_2^{\#\RI_\gp} \text{ or } \BZ_2^{\#\RI_\gp-1}.
\end{equation}
The latter case occurs if $G_n$ is even orthogonal and some orthogonal summand $\varphi_i$ for $i\in \RI_\gp$ has odd dimension, or if $G_n$ is symplectic.
As in \cite[(3.5)]{JLZ22}, for any $\varphi\in{\Phi}_\gen(G_n^*)$ we write elements of $\CS_\phi$ in the following form
\begin{equation}\label{eq:e-i}
(e_i)_{i\in \RI_{\gp}}\in \BZ_{2}^{\#\RI_{\gp}}, \text{ (or simply denoted by $(e_{i})$)}, \ {\rm with}\ e_i\in\{0,1\}, 
\end{equation}
where each $e_i$ corresponds to $\varphi_i$-component in the decomposition \eqref{edec} for $i\in \RI_\gp$.
In the case that $\wh{G}=\SO_N(\BC)$, denote by $A_\varphi$ the component group  ${\rm Cent}_{\RO_N(\BC)}(\varphi)/{\rm Cent}_{\RO_N(\BC)}(\varphi)^\circ$.
Then $\CS_\varphi$ consists of elements in $A_\varphi$ with determinant $1$ and is a subgroup of index $1$ or $2$.
Also write elements of $A_\varphi$ in the form
$(e_i)$ with $e_i\in\{0,1\}$ corresponding to the $\varphi_i$-component in the decomposition \eqref{edec} for $i\in \RI_\gp$.
When $G_n$ is even orthogonal and some $\varphi_i$ for $i\in \RI_\gp$ has odd dimension or $G_n$ is symplectic, $(e_i)_{i\in \RI_{\gp}}$ is in $\CS_\varphi$ if and only if $\sum_{i\in\RI_\gp}e_i\dim\varphi_i$ is even.

The Vogan version of the local Langlands correspondence for the classical groups under consideration asserts that for any $\varphi\in{\Phi}_\gen(G_n^*)$ the local Vogan packet 
${\Pi}_\varphi[G_n^*]$ is in one-to-one correspondence with the character group $\wh{\CS_\varphi}$ (the Pontryagin dual) of the component group $\CS_\varphi$. The parameterization depends on a 
choice of the Whittaker datum on $G_n^*(F)$, which as explained in \cite[Section 3.2]{JLZ22}, is completely determined by a number $a$ belonging to a group $\CZ$. As in \cite[Section 3.2.1]{JLZ22}, 
this group $\CZ$ is defined to be
\begin{equation} \label{CZ}
\CZ=F^\times/ \BN E^\times \cong
\begin{cases} 
F^\times/ F^{\times2}, & \textrm{if }E=F,\\
\textrm{Gal}(E/F), & \textrm{if }E=F(\delta).
\end{cases}
\end{equation}
We refer to \cite[Section 3.2.1]{JLZ22} for a detailed discussion of the $F$-rationality of the local Langlands correspondence and the relation between the 
Whittaker data on $G_n^*(F)$ and the $\CZ$-action on the character group $\wh{\CS_\varphi}$ of the component group $\CS_\varphi$. With a fixed $a\in\CZ$, which determines the local Whittaker datum and 
the $F$-rationality of the local Langlands correspondence, for any $\pi\in {\Pi}_\varphi[G_n^*]$, there exists a unique character $\mu\in\wh{\CS_\varphi}$, such that the enhanced $L$-parameter $(\varphi,\mu)$ corresponds to $\pi$. In this case, we write 
\begin{equation}\label{pia}
    \pi=\pi_a(\varphi,\mu),\quad \iota_a(\pi)=(\varphi,\mu).
\end{equation}

\subsection{Local Gan-Gross-Prasad conjecture}\label{ssec-LGGP}

The local Gan-Gross-Prasad conjecture for classical groups and generic $L$-parameters (\cite{GGP12}) is one of the main inputs in the theory of {\sl arithmetic wavefront sets} as 
developed in \cite{JLZ22} and in what we are discussing in this paper. 

It is important to point out that the local Gan-Gross-Prasad conjecture (\cite{GGP12})
is now a well-established theorem for classical groups and generic local $L$-parameters. More precisely, it is proved for all $p$-adic cases by J.-L.\ Waldspurger (\cite{W10, W12a, W12b}),
C. M\oe glin and J.-L.\ Waldspurger (\cite{MW12}), R. Beuzart-Plessis (\cite{BP16}),  W. T. Gan and A. Ichino (\cite{GI16}), and H. Atobe (\cite{At18}). It is proved for real unitary groups by R. Beuzart-Plessis \cite{BP20}, (H. He (\cite{He17}) for some special cases) and H. Xue (\cite{X1, X2}). Finally, it is proved for real special orthogonal groups  by  Z. Luo (\cite{Lu20}), C. Chen (\cite{Ch21}), and Chen-Luo (\cite{CL22}). Among many other results, the rank one real orthogonal groups are also treated in the work of T. Kobayashi and B. Speh \cite{KS15, KS18}.

For convenience, we recall from \cite[Section 4.4]{JLZ22} the explicit arithmetic data that we need for this paper from the local Gan-Gross-Prasad conjecture (\cite{GGP12}). They are given case by case, 
in order to take into account the compatibility of the local conjectures (\cite[\S18]{GGP12}) in our formulation through the rationality of nilpotent orbits. 

\subsubsection{Special orthogonal groups}\label{sec:GGP-SO}
Let $(\varphi, \phi) \in \Phi_\gen(\SO(V^*)\times \SO(W^*))$, where $V^*\supset W^*$ is a relevant pair of orthogonal spaces of dimensions $\Fn>\Fn'$ with opposite parities. Then the local Gan-Gross-Prasad conjecture asserts that up to isometry there is a unique relevant pure inner form $\SO(V)\times \SO(W)$ and  a unique pair of representations $(\pi, \pi')\in \Pi_\varphi(\SO(V))\times \Pi_\phi(\SO(W))$ up to $\sim_{\rm c}$ equivalence,  such that
\[
\Hom_{P_{X_{e,\varsigma}}}(\pi\wh{\otimes} \pi', \psi_{X_{e,\varsigma}})\neq 0,% \blue{(\varsigma \textrm{ irrelevant here?})}
\]
with  $p_1=\Fn-\Fn'$, $m=\lfloor p_1/2\rfloor$, $W^\perp=\BH^m+F e$, $X_{e, \varsigma}\in \CO^{\st}_{p_1}$
and $P_{X_{e,\varsigma}} = M_{X_{e,\varsigma}}^\circ\ltimes U_{X_{e, \varsigma}}$. Here and below, $\BH$ denotes a hyperbolic plane (over $E$).
Moreover, the pair $(\pi, \pi')$ is determined by
\[
\pi=\pi_a(\varphi, \chi_{\varphi,\phi}),\quad \pi'=\pi_a(\phi,\chi_{\phi,\varphi}),
\]
where $(\chi_{\varphi, \phi}, \chi_{\phi, \varphi})$ is given by \eqref{char-O}, and the weak local Langlands correspondence for the even orthogonal group is via $\iota_a$ as in \eqref{pia} with 
\[
a=-\disc(V)\disc(W)=(-1)^\Fn\langle e, e\rangle.
\]

\subsubsection{Symplectic-metaplectic groups}
Let $(\varphi, \phi) \in \Phi_\gen(\Sp(V)\times \Mp(W))$ or $\Phi_\gen(\Mp(V)\times \Sp(W))$, where $V\supset W$ is a pair of symplectic spaces of dimensions  $2n\geq 2n'$. Then the local Gan-Gross-Prasad conjecture asserts that for $\varsigma\in F^\times$,  there is  a unique pair of representations $(\pi, \pi')\in \Pi_\varphi(\Sp(V))\times \Pi_\phi(\Mp(W))$ or $\Pi_\varphi(\Mp(V))\times \Pi_\phi(\Sp(W))$ such that
\[
\Hom_{\wt{P}_{X_{e,\varsigma}}}(\pi\wh{\otimes} \pi', \omega_{\psi_{X_{e, \varsigma}}})\neq 0,
\]
with $p_1=2n-2n'$, $m=p_1/2$, $W^\perp =\BH^m$, $X_{e, \varsigma}\in \CO^{\st}_{p_1}$ 
and $\wt{P}_{X_{e,\varsigma}} = \wt{M}_{X_{e,\varsigma}}\ltimes U_{X_{e, \varsigma}}$.
Moreover, the pair $(\pi, \pi')$ is determined by
\[
\pi=\pi_1(\varphi, \chi^a_{\varphi,\phi}),\quad \pi'=\pi_1(\phi,\chi^a_{\phi,\varphi}),
\]
where  $(\chi^a_{\varphi, \phi}, \chi^a_{\phi, \varphi})$ is given by \eqref{char-spmp}, $a=(-1)^{m-1}\varsigma$, and the local Langlands correspondence is via $\iota_1$. The above is a reformulation of \cite[Prop. 18.1]{GGP12}.

\subsubsection{Unitary groups: Bessel case}
Let $(\varphi, \phi) \in \Phi_\gen(\RU(V^*)\times \RU(W^*))$, where $V^*\supset W^*$ is a relevant pair of Hermitian spaces of dimensions $\Fn>\Fn'$ with opposite parities. Then the local Gan-Gross-Prasad conjecture asserts that  for a trace zero element $\varsigma\in E^\times$, up to isometry there is a unique relevant pure inner form $\RU(V)\times \RU(W)$ and  a unique pair of representations $(\pi, \pi')\in \Pi_\varphi(\RU(V))\times \Pi_\phi(\RU(W))$,  such that
\[
\Hom_{P_{X_{e,\varsigma}}}(\pi\wh{\otimes} \pi', \psi_{X_{e,\varsigma}})\neq 0,
\]
with  $p_1=\Fn-\Fn'$, $m=\lfloor p_1/2\rfloor$, $W^\perp=\BH^m+E e$,  $X_{e, \varsigma}\in \CO^{\st}_{p_1}$
and $P_{X_{e,\varsigma}} = M_{X_{e,\varsigma}}\ltimes U_{X_{e, \varsigma}}$.
Moreover, the pair $(\pi, \pi')$ is determined by
\[
\pi=\pi_a(\varphi, \chi_{\varphi,\phi}),\quad \pi'=\pi_{-a}(\phi,\chi_{\phi,\varphi}),
\]
where $(\chi_{\varphi, \phi}, \chi_{\phi, \varphi})$ is given by \eqref{char-U}, and
\[
a=(-1)^{\Fn-1}\disc(V)\disc(W) = \langle e,e \rangle.
\]

\subsubsection{Unitary groups: Fourier-Jacobi case} Fix a splitting character $\xi$ that is used to define the Heisenberg-oscillator representation, namely $\xi$ is a  character of $E^\times$ such that $\xi|_{F^\times}=\omega_{E/F}$  is the quadratic character of $\CZ=F^\times/\BN E^\times$ given by the local class field theory.  Let $(\varphi, \phi) \in \Phi_\gen(\RU(V^*)\times \RU(W^*))$, where $V^*\supset W^*$ is a relevant pair of skew-Hermitian spaces of dimensions $\Fn \geq \Fn'$ with the same parity. Then the local Gan-Gross-Prasad conjecture asserts that for $\varsigma\in F^\times$,
up to isometry there is a unique relevant pure inner form $\RU(V)\times \RU(W)$ and  a unique pair of representations $(\pi, \pi')\in \Pi_\varphi(\RU(V))\times \Pi_\phi(\RU(W))$,  such that
\[
\Hom_{P_{X_{e,\varsigma}}}(\pi\wh{\otimes} \pi', \omega_{\psi_{X_{e,\varsigma}},\xi})\neq 0,
\]
with  $p_1=\Fn-\Fn'$, $m=p_1/2$, $W^\perp =\BH^m$, $X_{e, \varsigma}\in \CO^{\st}_{p_1}$ and $P_{X_{e,\varsigma}} = M_{X_{e,\varsigma}}\ltimes U_{X_{e, \varsigma}}$.
Moreover, the pair $(\pi, \pi')$ is determined by
\[
\pi=\pi_1(\varphi, \chi_{\varphi(\xi^{-1}),\phi}\cdot \eta_a),\quad \pi'=\pi_1(\phi,\chi_{\phi,\varphi(\xi^{-1})}\cdot\eta_a),
\]
where $(\chi_{\varphi(\xi^{-1}), \phi}, \chi_{\phi, \varphi(\xi^{-1})})$ is given by \eqref{char-U},
$a= (-1)^{m-1}\varsigma$,
and the local Langlands correspondence  is via $\iota_1$.

\subsection{Twisted distinguished characters} \label{ssec-DC}

Assume that $\varphi$ and $\phi$ are two (conjugate) self-dual representations
of the local Langlands group $\CL_E$ (see \eqref{LLG} for definition) of opposite types, written in the form \eqref{edec} with index sets $\RI_\gp$ and $\RI_\gp'$, respectively. 
The pair of twisted characters $(\chi_{\varphi,\phi},\chi_{\phi,\varphi})$ is the key arithmetic ingredient in the local Gan-Gross-Prasad conjecture as discussed in Section \ref{ssec-LGGP}.
We recall from \cite[Section 3.4]{JLZ22} the explicit expressions of $(\chi_{\varphi,\phi},\chi_{\phi,\varphi})$ in each case. 
We write the elements of the corresponding component groups $\CS_\varphi$ and $\CS_\phi$ as $(e_i)_{i\in \RI_\gp}$ and $(e_{i'})_{i'\in \RI_\gp'}$ respectively. 
Let $\psi_0$ be a fixed nontrivial additive character of $F$ and $\psi$ be the character of $E$ 
given by 
\begin{equation}\label{psiE}
\psi(x):=\psi_0(\frac{1}{2}\tr_{E/F}(\delta\cdot x))
\end{equation}
with $E=F(\delta)$ as in Section \ref{ssec-CGIF}. For $a\in \mathcal{Z}$, define a character $\eta_a$ of $\CS_\varphi$ by\footnote{For $\Mp_{2n}$ the character $\eta_a$ is denoted by $\eta[a]$ in \cite{GGP12}; we change the notation for uniformity.}
\begin{equation}\label{eta}
\eta_a((e_i)_{i\in \RI_{\gp}})=
\begin{cases} \prod_{i\in \RI_{\gp}}(\det \varphi_i, a)_F^{e_i}, & \textrm{if }E=F\textrm{ and }G_n^*\neq \Mp_{2n},\\
\prod_{i\in \RI_\gp} \left(\varepsilon(\varphi_i)\varepsilon(\varphi_i(a))(a,-1)_F^{\dim \phi_i/2}\right)^{e_i}, & \textrm{if }E=F\textrm{ and }G_n^*=\Mp_{2n},\\
\prod_{i\in \RI_{\gp}}\omega_{E/F}(a)^{e_i\dim \varphi_i}, & \textrm{if }E=F(\delta),
\end{cases}
\end{equation}
where $(\cdot, \cdot)_F$ is the Hilbert symbol defined over $F$, and $\omega_{E/F}$ is the quadratic character of $\CZ$ given by the local classfield theory. Note that $\eta_a$ is trivial if $G_n$ is special odd orthogonal.
Then we have a $\CZ$-action on $\widehat{\CS}_\varphi$ by
\begin{align}\label{Z-action}
a\ \colon\ \widehat{\CS_\varphi} \to \widehat{\CS_\varphi},\quad \mu\mapsto \mu\cdot\eta_a
\end{align}
for every $a\in\CZ$. It clear that one may define a character $\eta_a$ for $\phi$ and define 
an action of $\CZ$ on $\wh{\CS_\phi}$.

\subsubsection{}
If  $E=F$, $\varphi$ and $\phi$ are self-dual of even dimensions, define a pair of characters $(\chi_{\varphi,\phi}, \chi_{\phi,\varphi})\in \wh{\CS}_\varphi\times\wh{\CS}_\phi$ by
\begin{equation} \label{char-O}
 \chi_{\varphi,\phi}((e_i)_{i\in \RI_\gp}):= \prod_{i\in \RI_\gp}\left(\varepsilon(\varphi_i\otimes\phi) (\det \varphi_i)(-1)^{\dim \phi/2}(\det \phi)(-1)^{\dim\varphi_i/2}\right)^{e_i},
\end{equation}
and a similar formula defines $\chi_{\phi,\varphi}$. Note that the local root number
\[
\varepsilon(\varphi_i\otimes\phi):=\varepsilon(\varphi_i\otimes\phi, \psi_0)
\]
 does not depend on $\psi_0$. In this case, for $z\in \CZ$ define
 \begin{equation}
 (\chi_{\varphi, \phi}^z, \chi_{\phi,\varphi}^z) = (\chi_{\varphi, \phi}\cdot \eta_z, \chi_{\phi,\varphi}\cdot \eta_z),
 \end{equation}
 where $\eta_z$ is given by \eqref{eta}.

\subsubsection{}
If $E=F$, $\varphi$ is odd orthogonal and $\phi$ is symplectic,  define $(\chi_{\varphi,\phi}, \chi_{\phi,\varphi})$ to be the restriction of
\[
(\chi_{\varphi_1,\phi},\chi_{\phi,\varphi_1})\in
\wh{\CS}_{\varphi_1}\times\wh{\CS}_\phi
\]
to $\CS_{\varphi}\times \CS_\phi$, where
\[
\varphi_1:=\varphi\oplus\BC.
\]
We also consider the twist $\varphi(z)$ with $z\in \CZ$, and make the canonical identification $\CS_{\varphi(z)}\cong \CS_\varphi$. Then we define
\begin{equation} \label{char-spmp}
(\chi^z_{\varphi, \phi}, \chi^z_{\phi,\varphi})=(\chi_{\varphi(z), \phi} \cdot \eta_z, \chi_{\phi, \varphi(z)})\in \widehat{\CS}_\varphi\times\widehat{\CS}_\phi.
\end{equation}

\subsubsection{}
If $E=F(\delta)$, $\varphi$ and $\phi$ are conjugate self-dual, define $(\chi_{\varphi,\phi}, \chi_{\phi,\varphi})\in \wh{\CS}_\varphi\times\wh{\CS}_\phi$ by
\begin{equation} \label{char-U}
 \chi_{\varphi,\phi}((e_i)_{i\in \RI_\gp})= \prod_{i\in \RI_\gp}\varepsilon(\varphi_i\otimes\phi, \psi) ^{e_i},
\end{equation}
and a similar formula defines $\chi_{\phi,\varphi}$, where $\psi$ is given by \eqref{psiE}. Put $\Fl= \dim \varphi - \dim\phi$. In this case, for $z\in \CZ$ define
\begin{equation}\label{chi-a}
(\chi^z_{\varphi,\phi}, \chi^z_{\phi,\varphi})= (\chi_{\varphi, \phi}\cdot \eta_z, \chi_{\phi,\varphi}\cdot \eta_{(-1)^{\Fl}z}).
\end{equation}

By convention, in the above cases we allow $\varphi$ or $\phi$ to be zero. In such a case we formally take $\chi_{\varphi, \phi}$ and $\chi_{\phi,\varphi}$ to be the trivial characters of $\CS_\varphi$ and $\CS_\phi$ respectively.

\subsection{Descent of enhanced $L$-parameters}\label{ssec-DELP}

In order to unify the notations, we denote  
\[
\widetilde{\Phi}_\gen(G_n^*)=\Phi_\gen(G_n^*)
\]
if $E=F$, and 
\[
\widetilde{\Phi}_\gen(G_n^*)= \Phi_\gen(G_n^*)\cup\Phi_\gen(G_n^*)\otimes\xi
\]
the set of all $\frak{n}$-dimensional conjugate self-dual $L$-parameters if $E=F(\delta)$, where $\xi$ is any  character of $E^\times$ such that $\xi|_{F^\times}=\omega_{E/F}$. Note that if $E=F(\delta)$ then $\Phi_\gen(G_n^*)$ consists of conjugate self-dual $L$-parameters  of sign $(-1)^{\Fn-1}$. This modification is made to avoid the choice of a splitting character $\xi$ in the Fourier-Jacobi model of  unitary groups.

Let $\varphi\in\widetilde{\Phi}_\gen(G_n^*)$ be an $L$-parameter given by the form \eqref{edec}. For $\mu\in \widehat{\CS}_\varphi$, define its $\CZ$-orbit
\[
\ScO_\CZ(\mu)=\{\mu\cdot\eta_z \in \widehat{\CS}_\varphi \mid z\in\CZ\}.
\]
We recall from \cite[Section 5.1]{JLZ22} the definition and properties of the descents of enhanced $L$-parameters 
$(\varphi,\mu)$ with $\varphi\in \widetilde{\Phi}_\gen(G_n^*)$ and $\chi\in
\widehat{\CS}_\varphi$, along with an $\Fn$-dimensional $\epsilon$-Hermitian vector space $(V, q)$ over $F$.
As in Section \ref{ssec-CGIF}, let $(V, q)$ be an $\Fn$-dimensional $\epsilon$-Hermitian vector space over $F$. Let $0<\Fl\leq \Fn$ be such that
$
(-1)^{\Fl-1}=\epsilon
$
if $E=F$. There is no restriction on the parity of $\Fl$ if $E=F(\delta)$. 
We obtain quasi-split relevant pairs of classical groups 
$(G_n^*, H_{\lfloor(\frak{n}-\Fl)/2\rfloor}^*)$ as given in \eqref{relevantpair}.

\begin{defn}[Definition 5.1 of \cite{JLZ22}]\label{defn:pd}
For any given enhanced $L$-parameter $(\varphi,\mu)$ with
$\varphi\in \widetilde{\Phi}_\gen(G_n^*)$ and $\mu\in\widehat{\CS_\varphi}$,
the $\Fl$-th  descent $\FD_{\Fl}^z(\varphi,\mu)$ of $(\varphi,\mu)$ along $z\in \CZ$ is the set of contragredients $(\wh{\phi}, \wh{\nu})$ of all enhanced $L$-parameters $(\phi,\nu)$ with
$\phi\in \widetilde{\Phi}_\gen(H_{\lfloor(\frak{n}-\Fl)/2\rfloor}^*)$ and $\nu\in\wh{\CS_\phi}$ satisfying the following conditions:
\begin{enumerate}
\item  the type of $\phi$ is different from that of $\varphi$, and
\item the following
\[
(\chi^z_{\varphi,\phi}, \chi^z_{\phi, \chi} ) = (\mu, \nu),
\]
holds, where $(\chi^z_{\varphi,\phi},\chi^z_{\phi,\varphi})$ is defined, case by case, as in Section {\rm \ref{ssec-DC}}.
\end{enumerate}
The $\Fl$-th  descent of the enhanced $L$-parameter $(\varphi,\mu)$  is defined to be
\begin{equation}\label{Fl-pd}
\FD_\Fl(\varphi, \mu) := \bigcupdot_{z\in \CZ} \FD_\Fl^z(\varphi, \mu).
\end{equation}
\end{defn}
The contragredient $(\wh\phi,\wh{\nu})$ of $(\phi,\nu)$ is defined in 
\cite[Section 3.3]{JLZ22} and is explicitly given in \cite[Proposition 3.2]{JLZ22} for the 
classical groups under consideration. 

\begin{prop}[Proposition 3.2 of \cite{JLZ22}]\label{prop:dualdata}
For an enhanced $L$-parameter $(\varphi,\mu)$  with
$\varphi\in \widetilde{\Phi}_\gen(G_n^*)$ and $\mu\in\widehat{\CS_\varphi}$, 
the contragredient $(\wh{\varphi}, \wh{\mu})$ can be written, case by case, as follows:
\begin{itemize}
\item $G^*=\SO_\Fn^*$, $(\wh\varphi, \wh\mu) = (\varphi, \mu)$.
\item $G^*=\Sp_{2n}$, $(\wh\varphi, \wh\mu) = (\varphi, \mu \cdot \eta_{-1})$.
\item $G^*=\Mp_{2n}$,  $(\wh\varphi, \wh\mu) = (\varphi(-1), \mu\cdot \eta_{-1})$.
\item $G^* =\RU_\Fn^*$, $(\wh\varphi, \wh\mu) = (\varphi^\vee, \mu\cdot\eta)$ where $\eta$ is trivial if $\Fn$ is odd and $\eta = \eta_{-1}$ if $\Fn$ is even.
\end{itemize}
In all cases as listed above, the following equality 
\[
(\wh\varphi, \wh{\mu\cdot \eta_a}) = (\wh\varphi, \wh\mu\cdot \eta_a)
\]
holds for any $a\in \CZ$.
\end{prop}

We prove the tower property of the descents of enhanced $L$-parameters, which can be stated as follows.

\begin{prop}\label{prop:TP}
For any enhanced $L$-parameter $(\varphi,\mu)$ with $\varphi\in\widetilde{\Phi}_\gen(G_n^*)$ and $\mu\in\wh{\CS_\varphi}$, if there exists an integer $\Fl_1$ with $0<\Fl_1\leq\Fn$ such that the $\Fl_1$-th descent $\FD_{\Fl_1}(\varphi, \mu)$ is non-empty, then for any integer $\Fl$ 
with $0<\Fl\leq\Fl_1$ and $\Fl\equiv \Fl_1 \text{ mod }2$, the $\Fl$-th descent $\FD_\Fl(\varphi, \mu)$ is non-empty.
\end{prop}

\begin{proof}
Let $\phi_0$ be a tempered $L$-parameter of dimension $(\Fl_1-\Fl)/2$ which contains no irreducible (conjugate) self-dual summands of the same type as $\varphi$.
For instance, one can choose $\phi_0$ to be a direct sum of unitary characters of the Weil-Deligne group which are not quadratic and not conjugate self-dual. Assume that $(\wh\phi,\wh\nu)\in \FD_{\Fl_1}^z(\varphi,\mu)$.  
Let $\phi'=\phi\oplus\phi_0\oplus\phi_0^\vee$. Then $\CS_{\phi'}=\CS_\phi$ and it is clear that 
$(\wh{\phi'},\wh\nu)\in \FD_{\Fl}^z(\varphi,\mu)$. Thus $\FD_\Fl(\varphi, \mu)\neq\varnothing$.
\end{proof}

By \cite[Proposition 5.2]{JLZ22}, for any enhanced $L$-parameter $(\varphi,\mu)$ with $\varphi\in\widetilde{\Phi}_\gen(G_n^*)$ and $\mu\in\wh{\CS_\varphi}$, there exists an integer $\Fl$ with $0<\Fl\leq\Fn$ such that the $\Fl$-th descent $\FD_\Fl(\varphi, \mu)$ is non-empty. This leads to a definition of the {\sl first occurrence index} of the enhanced $L$-parameter 
$(\varphi, \mu)$, as in \cite[(5.3)]{JLZ22}, 
\begin{equation} \label{FO}
\Fl_0=\Fl_0(\varphi, \mu):=\max\{0<\Fl\leq \Fn \mid \FD_\Fl(\varphi, \mu)\neq \emptyset\};
\end{equation}
and the $\Fl_0$-th  descent $\FD_{\Fl_0}(\varphi, \mu)$ is called the {\it first descent} of the enhanced $L$-parameter $(\varphi,\mu)$. The following is the first important result on 
descents of enhanced $L$-parameters. 

\begin{thm}[Theorem 5.3 of \cite{JLZ22}]\label{thm:DFD}
For any enhanced $L$-parameter $(\varphi,\mu)$ with $\varphi\in\widetilde{\Phi}_\gen(G_n^*)$ and $\mu\in\wh{\CS_\varphi}$, the first descent $\FD_{\Fl_0}(\varphi, \mu)$ at 
the first occurrence index $\Fl_0=\Fl_0(\varphi, \mu)$ contains only of discrete enhanced $L$-parameters. 
\end{thm}

\subsection{Proof of Proposition \ref{prop:sne}}\label{ssec-PSNE}
We are ready to prove Proposition \ref{prop:sne}, that is, for any $\pi\in\Pi_\varphi[G_n^*]$ with $\varphi$ generic, there exists an integer $p_1$ with $0<p_1\leq\Fn$, such that 
the spectrum $\Spec_{\CO_{p_1}}(\pi)$ is non-empty for some $F$-rational nilpotent orbit 
$\CO_{p_1}$ in the $F$-stable nilpotent orbit $\CO_{p_1}^\st$ associated with the partition 
$[p_1,1^{\Fn-p_1}]$. 

By the Vogan version of the local Langlands correspondence $\iota_a$ for some fixed $a\in\CZ$, we write 
\[
\pi=\pi_a(\varphi,\mu)
\]
for some character $\mu\in\wh{\CS_\varphi}$, which is unique. By \cite[Proposition 5.2]{JLZ22}, there exists an integer $\Fl$ with $0<\Fl\leq\Fn$, such that the descent 
$\FD_\Fl(\varphi,\mu)$ is non empty. This means that there exists an enhanced $L$-parameter 
$(\phi,\nu)$ of $H_{\lfloor(\frak{n}-\Fl)/2\rfloor}^*(F)$ with $\phi$ generic, such that the contragredient 
$(\wh\phi,\wh\nu)$ belongs to the $\Fl$-th descent $\FD_\Fl(\varphi,\mu)$. 

From Definition \ref{defn:pd}, there exists a number $z\in\CZ$, such that 
\[
(\chi^z_{\varphi,\phi}, \chi^z_{\phi,\varphi})=(\mu,\nu),
\]
where $(\chi^z_{\varphi,\phi},\chi^z_{\phi,\varphi})$ is defined, case by case, as in Section {\rm \ref{ssec-DC}}. By the Vogan version of the local Langlands correspondence, there 
exists an $F$-pure inner form $H_{\lfloor(\frak{n}-\Fl)/2\rfloor}$ of 
$H_{\lfloor(\frak{n}-\Fl)/2\rfloor}^*$ and a 
$\sigma\in\Pi_\phi(H_{\lfloor(\frak{n}-\Fl)/2\rfloor})$, such that the pair 
$(\pi,\sigma)$ is the distinguished pair in the local Vogan packet 
$\Pi_{\varphi\otimes\phi}[G_n^*\times H_{\lfloor(\frak{n}-\Fl)/2\rfloor}^*]$. Hence there exists an $F$-rational nilpotent orbit $\CO_\Fl$, associated with the partition 
$[\Fl,1^{\Fn-\Fl}]$, such that $\sigma^\vee$ belongs to the spectrum 
$\Spec_{\CO_\Fl}(\pi)$. Finally, by taking $p_1=\Fl$, we finish the proof of 
Proposition \ref{prop:sne}.

\quad

From the proof of Proposition \ref{prop:sne}, we obtain a stronger version of 
Proposition \ref{prop:sne}.

\begin{cor}\label{cor:sne}
If a representation $\pi\in\Pi_F(G_n)$ has a generic local $L$-parameter, then there exists at least one integer $p_1$ with $0<p_1\leq\Fn$, such that the spectrum $\Spec_{p_1}(\pi)$ 
contains at least one $\sigma^\vee$, which has a generic $L$-parameter. 
\end{cor}

By the relation between parabolic induction and the multiplicities, we are able to prove 
the following even stronger result. 

\begin{prop}\label{prop:stempered}
If a representation $\pi\in\Pi_F(G_n)$ has a generic local $L$-parameter, then there exists at least one integer $p_1$ with $0<p_1\leq\Fn$, such that the spectrum $\Spec_{p_1}(\pi)$ 
contains at least one $\sigma^\vee$, which is tempered. 
\end{prop}

\begin{proof}
This follows directly from Propositions \ref{prop:sne} and \ref{prop:PI}. 
\end{proof}

\section{Spectrum and Arithmetic First Descent}\label{sec-SDAFD}

%%%%%%%%%%%%%%%%%%%%%%%%%%%%%%%%%%%%%%%%%%%%%%%%%%%%%%%

\subsection{Arithmetic first descent}\label{ssec-afd}

We take the local Langlands correspondence $\iota_a$ with the $F$-rationality determined by $a\in\CZ$. For any generic $L$-parameter $\varphi\in\widetilde{\Phi}_\gen(G_n^*)$, any member $\pi\in\Pi_F[G_n^*]$ corresponds with a unique $\mu\in\wh{\CS_\varphi}$ such that $\pi=\pi_a(\varphi,\mu)$. By \cite[Proposition 5.2]{JLZ22}, the first occurrence index 
$\Fl_0=\Fl_0(\varphi,\mu)$ of the enhanced $L$-parameter $(\varphi,\mu)$ exists. 
In this case, we define 
the {\sl arithmetic first occurrence index} $\Ff_\Fa(\pi)$ of $\pi$ to be the 
first occurrence index $\Fl_0=\Fl_0(\varphi,\mu)$ of $(\varphi,\mu)$:
\begin{equation}\label{afoi}
    \Ff_\Fa(\pi):=\Fl_0(\varphi,\mu).
\end{equation}
The spectral descent $\Spec_{\Ff_\Fa(\pi)}(\pi)$ of $\pi$ at the arithmetic first occurrence index $\Ff_\Fa(\pi)$ is called the {\sl arithmetic first descent} of $\pi$.

\subsection{Discreteness of arithmetic first descent}\label{ssec-AFD}

By Theorem \ref{thm:DFD} and the local Gan-Gross-Prasad conjecture, we obtain the following discreteness result. 

\begin{thm}[Discreteness]\label{thm:AFD}
Given any generic $L$-parameter $\varphi\in\widetilde{\Phi}_\gen(G_n^*)$, for any 
$\pi\in\Pi_\varphi[G_n^*]$ with $\pi=\pi_a(\varphi,\mu)$ for some $\mu\in\wh{\CS_\varphi}$, the arithmetic first descent $\Spec_{\Ff_\Fa(\pi)}(\pi)$ consists exactly of the  
discrete series representations $\sigma^\vee$ whose enhanced $L$-parameters are given in the first descent $\FD_{\Fl_0}(\varphi, \mu)$ at 
the first occurrence index $\Fl_0=\Fl_0(\varphi, \mu)$. 
\end{thm}

\begin{proof}
For any $\sigma^\vee\in \Spec_{\Ff_\Fa(\pi)}(\pi)$ with $\Ff_\Fa(\pi)=\Fl_0(\varphi, \mu)$, there exists an $F$-rational nilpotent orbit $\CO_{\Ff_\Fa(\pi)}$, such that $\sigma^\vee\in\Spec_{\CO_{\Ff_\Fa(\pi)}}(\pi)$.
By definition, we must have that 
\[
\Hom_{H_{\CO_{\Ff_\Fa(\pi)}}(F)}(\CJ_{\CO_{\Ff_\Fa(\pi)}}(\pi)\,\wh\otimes\,\sigma, \mathbbm{1})\neq 0. 
\]
We realize $\sigma$ as the Langlands quotient 
$\CL(\udl{s},\tau_1,\dots,\tau_t,\sigma_0)$
of the standard module 
$\RI(\udl{s},\tau_1,\dots\tau_t,\sigma_0)$ as displayed in \eqref{SM1}, and obtain that 
\begin{equation} \label{pia3}
\Hom_{H_{\CO_{\Ff_\Fa(\pi)}}(F)}(\CJ_{\CO_{\Ff_\Fa(\pi)}}(\pi)\,\wh\otimes\,\RI(\udl{s},\tau_1,\dots\tau_t,\sigma_0)
, \mathbbm{1})\neq 0, 
\end{equation}
for the $F$-rational nilpotent orbit $\CO_{\Ff_\Fs(\pi)}$. 
By Proposition \ref{prop:PI}, we obtain that 
\[
    \Hom_{H_{\CO_{\Ff_\Fa(\pi)+2p_0}}(F)}(\CJ_{\CO_{\Ff_\Fa(\pi)+2p_0}}(\pi)\,\wh\otimes\,\sigma_0, \mathbbm{1})\neq 0
\]
for some $F$-rational nilpotent orbit $\CO_{\Ff_\Fa(\pi)+2p_0}$, which implies that the irreducible tempered representation 
$\sigma_0^\vee$ belongs to the spectrum $\Spec_{\CO_{\Ff_\Fa(\pi)+2p_0}}(\pi)$. By the Vogan version of the local Langlands correspondence, we write 
\[
\sigma_0=\sigma_a(\phi,\nu)
\]
for some generic (tempered) $L$-parameter $\phi$ of $H_{\lfloor(\frak{n}-\Ff_\Fa(\pi)-2p_0)/2\rfloor}^*$ and $\nu\in\wh{\CS_\phi}$. 
Since $(\pi,\sigma_0)$ is the unique distinguished pair in the local Vogan packet 
$\Pi_{\varphi\otimes\phi}[G_n^*\times H_{\lfloor(\frak{n}-\Ff_\Fa(\pi)-2p_0)/2\rfloor}^*]$, we must have, from the descent of the enhanced $L$-parameter $(\varphi,\mu)$ (Definition \ref{defn:pd}), that $(\wh\phi,\wh\nu)$ belongs to the $(\Ff_\Fa(\pi)+2p_0)$-th descent $\FD_{\Ff_\Fa(\pi)+2p_0}(\varphi,\mu)$ of 
$(\varphi,\mu)$. Since $\Ff_\Fa(\pi)=\Fl_0(\varphi, \mu)$ is the first occurrence index of 
the enhanced $L$-parameter $(\varphi,\mu)$, we must have that 
\[
\Ff_\Fa(\pi)+2p_0\leq\Ff_\Fa(\pi)=\Fl_0(\varphi, \mu).
\]
Hence we obtain that $p_0=0$ and $\sigma=\sigma_0$ is tempered. Moreover, we have that 
$(\wh\phi,\wh\nu)$ belongs to the first descent $\FD_{\Ff_\Fa(\pi)}(\varphi,\mu)$ of 
$(\varphi,\mu)$. By Theorem \ref{thm:DFD}, the enhanced $L$-parameter $(\phi,\nu)$ must be 
discrete, and hence $\sigma$ must be a discrete series representation. 

Conversely, if $\sigma$ is a discrete series representation with enhanced $L$-parameter $(\phi,\nu)$, whose contragredient $(\wh\phi,\wh\nu)$ belongs to the first descent $\FD_{\Fl_0}(\varphi, \mu)$ at 
the first occurrence index $\Fl_0=\Fl_0(\varphi, \mu)$, we are going to show that 
$\sigma^\vee\in\Spec_{\Ff_\Fa(\pi)}(\pi)$. 
%for some $F$-rational nilpotent orbit $\CO_{\Ff_\Fa(\pi)}$. 
In fact, since the discrete enhanced $L$-parameter $(\phi,\nu)$ of $\sigma$ has the property that its dual $(\wh\phi,\wh\nu)$ belongs to the first descent $\FD_{\Fl_0}(\varphi, \mu)$, 
%by Theorem \ref{thm:DFD}, $(\phi,\nu)$ must be a discrete enhanced $L$-parameter. 
by the definition of descent of enhanced $L$-parameters (Definition \ref{defn:pd}) and the local 
Gan-Gross-Prasad conjecture, $(\pi,\sigma)$ must be the distinguished pair in the local 
Vogan packet $\Pi_{\varphi\otimes\phi}[G_n^*\times H_{\lfloor(\frak{n}-\Ff_\Fa(\pi))/2\rfloor}^*]$. Thus, $\sigma^\vee\in\Spec_{\CO_{\Ff_\Fa(\pi)}}(\pi)$ for some $F$-rational nilpotent orbit 
$\CO_{\Ff_\Fa(\pi)}$. This proves that $\sigma^\vee\in\Spec_{\Ff_\Fa(\pi)}(\pi)$.
\end{proof}

\subsection{Proof of Theorem  \ref{thm:SA-foi}}\label{sec-PC-SAFOI}

Theorem \ref{thm:SA-foi} asserts that for any given generic $L$-parameter $\varphi$ of $G_n^*(F)$ and for any $\pi\in\Pi_\varphi[G_n^*]$, if $\pi$ has its enhanced $L$-parameter 
$(\varphi,\mu)$, i.e. 
\begin{equation}\label{pia1}
    \pi=\pi_a(\varphi,\mu)
\end{equation}
holds for some $\mu\in\wh{\CS_\varphi}$ (which is unique), then the spectral first occurrence index $\Ff_\Fs(\pi)$ is equal to the arithmetic first occurrence index $\Ff_\Fa(\pi)$, i.e.
\begin{equation}\label{sa-foi}
\Ff_\Fs(\pi)=\Ff_\Fa(\pi),
\end{equation}
where $\Ff_\Fa(\pi):=\Fl_0(\varphi,\mu)$, the first occurrence index of $(\varphi,\mu)$.
We are going to prove the identity in \eqref{sa-foi}. 

First, we prove that $\Ff_\Fa(\pi)\leq\Ff_\Fs(\pi)$. For the given $(\varphi,\mu)$ as in \eqref{pia1}, the first occurrence index $\Fl_0=\Fl_0(\varphi,\mu)$, which is equal to $\Ff_\Fa(\pi)$, is an integer with 
$0<\Fl_0\leq\Fn$. Take an enhanced $L$-parameter $(\phi,\nu)$ with $\phi$ generic, such that 
\[
(\wh\phi,\wh\nu)\in\FD_{\Fl_0}(\varphi,\mu).
\]
As in the proof of Proposition \ref{prop:sne}, the contragredient $\sigma^\vee$ of the representation
\[
\sigma=\sigma_a(\phi,\nu)
\]
belongs to the spectrum $\Spec_{\CO_{\Fl_0}}(\pi)$. In particular, the spectrum 
$\Spec_{\Fl_0}(\pi)$ is not empty. By definition of the spectral first occurrence index, we must have 
\begin{equation}\label{pia2}
    \Ff_\Fa(\pi)=\Fl_0(\varphi,\mu)\leq\Ff_\Fs(\pi). 
\end{equation}

Now we prove the converse: $\Ff_\Fa(\pi)\geq\Ff_\Fs(\pi)$. If $\sigma^\vee\in\Spec_{\Ff_\Fs(\pi)}(\pi)$, then there exists an $F$-rational nilpotent orbit $\CO_{\Ff_\Fs(\pi)}$, such that $\sigma^\vee\in\Spec_{\CO_{\Ff_\Fs(\pi)}}(\pi)$. By definition, we must have that 
\[
\Hom_{H_{\CO_{\Ff_\Fs(\pi)}}(F)}(\CJ_{\CO_{\Ff_\Fs(\pi)}}(\pi)\,\wh\otimes\,\sigma, \mathbbm{1})\neq 0. 
\]
We realize $\sigma$ as the Langlands quotient 
$\CL(\udl{s},\tau_1,\dots,\tau_t,\sigma_0)$
of the standard module 
$\RI(\udl{s},\tau_1,\dots\tau_t,\sigma_0)$ as in \eqref{SM1}, we obtain that 
\begin{equation} \label{pia3}
\Hom_{H_{\CO_{\Ff_\Fs(\pi)}}(F)}(\CJ_{\CO_{\Ff_\Fs(\pi)}}(\pi)\,\wh\otimes\,\RI(\udl{s},\tau_1,\dots\tau_t,\sigma_0)
, \mathbbm{1})\neq 0, 
\end{equation}
for the $F$-rational nilpotent orbit $\CO_{\Ff_\Fs(\pi)}$. 
By Proposition \ref{prop:PI}, we obtain that 
\[
    \Hom_{H_{\CO_{\Ff_\Fs(\pi)+2p_0}}(F)}(\CJ_{\CO_{\Ff_\Fs(\pi)+2p_0}}(\pi)\,\wh\otimes\,\sigma_0,\mathbbm{1})\neq 0
\]
for some $F$-rational nilpotent orbit $\CO_{\Ff_\Fs(\pi)+2p_0}$, which implies that the irreducible tempered representation 
$\sigma_0^\vee$ belongs to the spectrum $\Spec_{\CO_{\Ff_\Fs(\pi)+2p_0}}(\pi)$, and in particular the spectrum $\Spec_{\Ff_\Fs(\pi)+2p_0}(\pi)$ is not empty. 
Since $\Ff_\Fs(\pi)$ is the spectral first occurrence index of $\pi$, we must have 
\[
\Ff_\Fs(\pi)\geq \Ff_\Fs(\pi)+2p_0, 
\]
which implies that $p_0=0$, and $\sigma=\sigma_0$ is tempered. By the Vogan version of the 
local Langlands correspondence, we obtain that 
\[
\sigma=\sigma_a(\phi,\nu)
\]
for some generic (tempered) $L$-parameter $\phi$ of $H_{\lfloor(\frak{n}-\Ff_\Fs(\pi))/2\rfloor}^*$ and $\nu\in\wh{\CS_\phi}$. 
Since $(\pi,\sigma)$ is the unique distinguished pair in the local Vogan packet 
$\Pi_{\varphi\otimes\phi}[G_n^*\times H_{\lfloor(\frak{n}-\Ff_\Fs(\pi))/2\rfloor}^*]$, we must have, from the descent of the enhanced $L$-parameter $(\varphi,\mu)$ (Definition \ref{defn:pd}), that $(\wh\phi,\wh\nu)$ belongs to the $\Ff_\Fs(\pi)$-th descent $\FD_{\Ff_\Fs(\pi)}(\varphi,\mu)$ of 
$(\varphi,\mu)$. By the definition of the first occurrence index $\Fl_0(\varphi,\mu)$, 
we obtain that 
\[
\Ff_\Fs(\pi)\leq \Fl_0(\varphi,\mu)=\Ff_\Fa(\pi).
\]
Combining with \eqref{pia2}, we prove that $\Ff_\Fs(\pi)=\Ff_\Fa(\pi)$. This proves 
Theorem \ref{thm:SA-foi}.

It is important to mention that Theorems \ref{thm:SA-foi} and \ref{thm:AFD} imply the main result of this paper (Theorem \ref{thm:ADS}), which extends the main result (\cite[Theorem 1.7]{JZ18}) to the great generality.

\subsection{Submodule Theorem}\label{ssec-ST}

We are going to prove Theorem \ref{thm:SM} in this section. 

For any given $\pi\in\Pi_F(G_n)$ with a generic $L$-parameter, By Theorems 
\ref{thm:SA-foi} and \ref{thm:AFD}, we have that by taking 
\[
p_1:=\Ff_\Fs(\pi)=\Ff_\Fa(\pi)
\]
there exists an irreducible discrete series representation $\sigma$ of $H_{\CO_{p_1}}$ 
for some $F$-rational nilpotent orbit $\CO_{p_1}$, such that 
$\sigma^\vee$ belongs to the spectrum $\Spec_{\CO_{p_1}}(\pi)$. This means that 
\[
\Hom_{H_{\CO_{p_1}}(F)}(\CJ_{\CO_{p_1}}(\pi)\,\wh\otimes\,\sigma, \mathbbm{1})\neq 0. 
\]

\begin{lem}\label{frob}
For any $\pi\in\Pi_F(G_n)$ with a generic $L$-parameter, if $\sigma$ is an irreducible smooth representation of $H_{\CO_{p_1}}(F)$, then 
\[
\Hom_{H_{\CO_{p_1}}(F)}(\CJ_{\CO_{p_1}}(\pi)\,\wh\otimes\,\sigma,\mathbbm{1})
\cong
\Hom_{H_{\CO_{p_1}}(F)}(\CJ_{\CO_{p_1}}(\pi),\sigma^\vee)
\]
where $\sigma^\vee$ is the contragredient of $\sigma$. 
\end{lem}

\begin{proof}
The lemma is clear when $F$ is non-archimedean. For $F$ archimedean, it follows from \cite[Lemma 2.2.22]{Ch23}.
\end{proof}

% Since the irreducible discrete series representation $\sigma$ is unitary, 
By Lemma \ref{frob}, we have that 
\[
\Hom_{H_{\CO_{p_1}}(F)}(\CJ_{\CO_{p_1}}(\pi),\sigma^\vee)\neq 0. 
\]
Hence we have 
\[
\Hom_{G_n(F)}(\pi,\Ind^{G_n(F)}_{R_{\CO_{p_1}}(F)}(\sigma^\vee\otimes\psi_{\CO_{p_1}}))\neq 0. 
\]
We are done.

\section{Proof of Proposition \ref{prop:PI}}\label{sec-PPI}
%\begin{enumerate}
%    \item Notation $\geq,\geq,\boxtimes$
 %   \item $\psi_F$ appeared before?
  %  \item new notation $G^{(i)}_{X_{m_1'}}$ 
  %  \item $\Fn^+$ and $n^+$ in unitary cases?
%\end{enumerate}
%%%%%%%%%%%%%%%%%%%%%%%%%%%%%%%%%%%%%%%%%%%%%

We are going to complete the proof of Proposition \ref{prop:PI} in this section by establishing a multiplicity formula. For any $\pi\in \Pi_F(G_n)$ and $\sigma\in \Pi_F(H_{\lfloor (\Fn-p_1)/2 \rfloor})$, we define the multiplicity $m(\pi,\sigma)$ or $m(\sigma,\pi)$
the dimension
\begin{equation}\label{equ: def of multiplicity}
\dim\Hom_{H_{\CO_{p_1} }(F)}(\CJ_{\CO_{p_1}}(\pi)\,\wh{\otimes}\,\sigma, \mathbbm{1}).
\end{equation}
In order to prove Proposition \ref{prop:PI}, it suffices to show that when $\pi$ is in a generic $L$-packet and  $\sigma=\RI(\udl{s},\tau_1,\dots\tau_t,\sigma_0)$, where $\tau_i,\sigma_0$ are tempered representations, we have
\begin{equation}\label{equ: multiplicity ineq}
m(\pi,\sigma)\leq m(\pi,\sigma_0).    
\end{equation}

When $F$ is non-archimedean, it was proved for special orthogonal groups in \cite{MW12}, for unitary groups in \cite{GI16}, and for symplectic groups and metaplectic groups in \cite{At18} that for every representation $\pi$ in a generic $L$-packet, we have
\[
\pi=\RI(\udl{s}',\rho_1,\dots\rho_{t'},\pi_0)
\] 
with certain $\udl{s}'\in \BR^{t'}$ and tempered $\rho_i$, $\pi_0$. When $F$ is archimedean, this result was shown in \cite[\S 1.1]{X2}, \cite[\S 4]{Ch21} and \cite[Theorem 5.2.1]{Ch23}. With this classification,  the following proposition implies the inequality of multiplicities in \eqref{equ: multiplicity ineq}.

\begin{prop}\label{pro: multiplicity} 
Let $\udl{s}=(s_1,\cdots,s_t)\in \BR^t$ and $\udl{s}'=(s_1',\cdots,s_{t'}')\in \BR^{t'}$, satisfying 
\[
s_1\geq \cdots \geq s_t\geq 0\quad\text{and}  \quad
s_1'\geq \cdots \geq s_{t'}'\geq 0.
\]
 Suppose that $\rho_i,\tau_i$ are tempered representations of $\GL_{m_i}(E)$, $\GL_{m_i'}(E)$  and $\pi_0,\sigma_0$ are tempered representations of $G_{n},H_{\lfloor (n-p_1)/2 \rfloor}$ respectively. Then for 
 \[
 \pi=\RI(\udl{s},\rho_1,\dots\rho_t,\pi_0)\quad \text{and} \quad \sigma=\RI(\udl{s}',\tau_1,\dots\tau_{t'},\sigma_0),
 \]
 we have that
\begin{equation}\label{equ: multiplicity formula} 
m(\pi,\sigma)=m(\pi_0,\sigma_0).
\end{equation}
\end{prop}

In the setting of Proposition \ref{pro: multiplicity}, we denote by
\[
n^+=n+\sum_{i=1}^t{m_i},\quad n^{\prime,+}=\lfloor (\Fn-p_1)/2\rfloor+\sum_{i=1}^{t'}m_i'.
\]
When applied to the setting of Proposition \ref{prop:PI}, we always have that 
\[
%n^+\geq n^{\prime,+}\text{ in the Bessel cases},\quad
n^+\neq n^{\prime,+} \text{ in the Fourier-Jacobi cases}. 
\]
%In these situations, the multiplicity $m(\pi,\sigma)$ is defined by \eqref{equ: def of multiplicity}. 
It is worth mentioning that Proposition \ref{pro: multiplicity} indeed holds in a more general context, and we implement the definition for $m(\pi,\sigma)$ in the Fourier-Jacobi case with $n^+=n^{\prime,+}$ as following 
\[
m(\pi,\sigma)=\dim \Hom_{G_{n^+}(F)}(\pi \,\wh{\otimes}\, \sigma \,\wh{\otimes}\, \omega_{n^+,\psi_F},\mathbbm{1}).
\]
 Here the representation $\omega_{n^+,\psi_F}$ is the Heisenberg-oscillator representation of 
\begin{equation}\label{equ: define wtG}
\wt{G}^J_{n^+}:=\wt{G}_{n^+}\ltimes \CH_{n^+},\quad \text{ where } \wt{G}_{n^+}:=\begin{cases}
\RU_{\Fn^+}& \text{ when }G_{n^+}=\RU_{\Fn^+} \text{ and }n^+=\lfloor\Fn^+/2\rfloor,\\
 \Mp_{2n^+}& \text{ when }G_{n^+}=\Sp_{2n^+},\Mp_{2n^+}
\end{cases}
\end{equation}
associated to $\psi_F$, where $\CH_{n^+}$ is the Heisenberg group $\Res_{E/F}V\oplus \BG_{a,F}$ for the $\epsilon$-Hermitian space $V$ associated to $G_{n^+}$, and $\pi, \sigma$ in the completed tensor product are inflations of the corresponding representations of $G_{n^+},H_{n^+}$ to $\wt{G}_{n^+}^J$ respectively.

We first point out that Proposition \ref{pro: multiplicity} is known in the following situations. 
When $F$ is non-archimedean, Propostion \ref{pro: multiplicity} follows from \cite[Proposition 1.3]{MW12} when $G_n, H_{\lfloor (\Fn-p_1)/2 \rfloor}$ are special orthogonal groups, and it follows from \cite[Proposition 9.4]{GI16} when $G_n, H_{\lfloor (\Fn-p_1)/2 \rfloor}$ are unitary groups. When $F$ is archimedean, Proposition \ref{pro: multiplicity} follows from \cite[Theorem A]{Ch23}. 
Hence, it remains to prove Proposition \ref{pro: multiplicity} in
the Fourier-Jacobi cases when $F$ is non-archimedean. We present a proof of this situation by following \cite{MW12} and using the tempered local Gan-Gross-Prasad conjecture as proved in \cite{GI16, At18}.

 In the setting of Proposition \ref{pro: multiplicity}, we call the inequality
\begin{equation}\label{equ: first inequality}
m(\pi,\sigma)\leq m(\pi_0,\sigma_0)
\end{equation}
``the first inequality" and call the inequality 
\begin{equation}\label{equ: second inequality}
m(\pi,\sigma)\geq m(\pi_0,\sigma_0)
\end{equation}
``the second inequality".

Following the framework of M\oe glin and Waldspurger in \cite{MW12}, there are three key steps in the proof of the inequalities: reduction to basic cases, basic forms of the first inequality, and basic forms of the second inequality.

In the setting of non-archimedean Fourier-Jacobi cases, we state them in the following proposition. 

\begin{prop}\label{prop: three ingredients}
In the setting of Proposition \ref{pro: multiplicity}, we denote by
\[
\sigma_{t'-1}=\RI((s_2',\cdots,s_{t'}'),\tau_2,\dots,\tau_{t'},\sigma_0).
\]
Then we have the following results.
\begin{enumerate}
    \item {\bf (Reduction to basic cases)} When 
$n^+=n^{\prime,+}$ and $\tau_1$ is a supercuspidal representation, then
\[
m(\pi,\sigma)=m(\pi,\sigma_{t'-1})
\]
if  $\pi^{\vee}$ does not belong to the Bernstein component associated to $|\det|^{s_1'}\tau_1 \otimes \sigma'$, where $\sigma'$ is
any supercuspidal representation of a Levi subgroup of $H_{n^{\prime,+}-m_1'}$.
    \item {\bf (Basic forms of the first inequality)} When $n^+=n^{\prime,+}$ and $s_1'\geq s_1$,
\[
m(\pi,\sigma)\leq m(\pi,\sigma_{t'-1}).
\]
    \item {\bf (Basic forms of the second inequality)} When $n^+=n^{\prime,+}$,
\[
m(\pi,\sigma)\geq m(\pi_0,\sigma_0).
\]
\end{enumerate}    
\end{prop}

Since Part (1) was proved in \cite[Theorem 16.1]{GGP12}, we prove parts (2) and (3) in the section. We will apply Parts (1) and (2) to prove the first inequality of Proposition \ref{pro: multiplicity} and apply Parts (1) and (3) to prove the second inequality of Proposition \ref{pro: multiplicity}.

\subsection{The first inequality}
In this section, we aim to prove the first inequality \eqref{equ: first inequality} for Fourier-Jacobi models over non-archimedean local fields. 
\subsubsection{Basic forms of the first inequality}
First, we prove Part (2) of Proposition \ref{prop: three ingredients} following \cite[\S 1.4]{MW12}. See \cite[\S 4.1]{Ch23} for the archimedean counterpart. 
For this purpose, we study $m(\pi,\sigma)$ with $n^+=n^{\prime,+}$. Recall that in this situation,
\[
m(\pi,\sigma)=\dim \Hom_{G_{n^+}(F)}(\pi \otimes \sigma \otimes \omega_{n^+,\psi_F}, \mathbbm{1}).
\]

Since $n^+=n^{\prime,+}$, $G_{n^+}$ and $H_{n^{\prime,+}}$ are defined as the isometric groups on the same $\epsilon$-hermitian space $V$. We fix an $m_1'$-dimensional totally isotropic space $X_{m_1'}$ over $E$.  We denote by $P_{X_{m_1'}}$, $P'_{X_{m_1'}}$ and $\wt{P}_{X_{m_1'}}$ the  respective parabolic subgroups of $G_{n^+}$, $H_{n^{\prime,+}}$ and $\wt{G}_{n^+}$ stabilizing $X_{m_1'}'$. The Levi decomposition gives $P_{X_{m_1'}}=L_{X_{m_1'}}\ltimes N_{X_{m_1'}}$ with $L_{X_{m_1'}}=\Res_{E/F}\GL_{m_1'}\times G_{n^+-m_1'}$.
For $i=1,\cdots,m_1'$, we let $\GL^{(i)}_{m_1'}$ be the semidirect product of $\Res_{E/F}\GL(X_{m_1'-i})$ and
the unipotent part $N^{(i)}_{m_1'}$ of the parabolic subgroup of $\Res_{E/F}\GL(X_{m_1'})$ stabilizing $X_{m_1'-i}\subset X_{m_1'}$ and a full flag of $X_{m_1'}/X_{m_1'-i}$.

We let 
\[
P_{X_{m_1'}}^{(i)}=(\GL_{m_1'}^{(i)}\times G_{n^+-m_1'})\ltimes N_{X_{m_1'}}^{(i)}\subset P_{X_{m_1'}}.
\]
%Let $N_{X_{m_1'}}^{(i)}$ be the unipotent part of the parabolic subgroup of $G_{n^+}$ stabilizing $X_{m_1'-i}\subset X_{m_1'}$ and a full flag of $X_{m_1'}/X_{m_1'-i}$, and 
%\begin{equation}\label{equ: derivative parabolic groups}
%P_{X_{m_1'}}^{(i)}=(\Res_{E/F}\GL_{m_1'-i}\times 1^i\times G_{n^+-m_1'-i})\ltimes N_{X_{m_1'}}^{(i)}.
%\end{equation}
%Then we obtain that 
%\begin{equation}\label{equ: multiplicity small}
%\begin{aligned}
%m(\pi,\sigma_{t'-1})&=\dim \Hom_{G_{n^+}(F)}(\ind_{P_{X_{m_1'}}^{(m_1')}}^{G_{n^+}}(\pi|_{P_{X_{m_1'}}^{(m_1')}}\wh{\otimes} \sigma_{t'-1}\wh{\otimes}\omega_{n^+-m_1',\psi_F}\wh{\otimes} \psi_{m_1'}^{-1}),\mathbbm{1}) \\
%&=\dim \Hom_{G_{n^+}(F)}(\pi\wh{\otimes}\ind_{P_{X_{m_1'}}^{(m_1')}}^{G_{n^+}}( \sigma_{t
%-1}\wh{\otimes}\omega_{n^+-m_1',\psi_F}\wh{\otimes} \psi_{m_1'}^{-1}), \mathbbm{1}).
%\end{aligned}
%\end{equation}
%{\color{blue}\begin{defn}
%Following the notations in [HS 1768], we study semialgebraic 
%    \begin{enumerate}
%    \item $\CE$
%        \item We denote by $\Gamma^{\CS}(\CX ,\CE)$ the space of Schwartz section
%        \item We denote by 
%        \[
%        \ind_H^G(\sigma)=\Gamma^{\CS}(H\bs G,G\times_H \sigma)
%        \]
%        In particular, from the Sch\"odinger's model of the Weil representation, we have
%        \[
%        \omega_{n^+,\psi_F}=\ind_{X_{n^+}\oplus \BG_a}^{\CH_{n^+}}(\psi_F)
%        \]
%    \end{enumerate}
%\end{defn}}
By definition, we have 
\[ 
\sigma =\RI(\udl{s}',\tau_1,\dots,\tau_{t'},\sigma_0)=\ind_{P_{X_{m_1'}}'(F)}^{H_{n^{\prime,+}}(F)}(\delta_{P_{X_{m_1'}}'}^{1/2}|\det|^{s_1'}\tau_1 \otimes \sigma_{t-1}).
\]
Here and thereafter, we denote by $\ind_{H}^{G}$ the (unnormalized) compact induction from $H$ to $G$. 
From the mixed model of the Weil representation, we have
\[\omega_{n^+,\psi_F}=\ind^{\wt{G}_{n^+}\ltimes \CH_{n^+}(F)}_{\wt{P}_{X_{m_1'}}\ltimes \CH(X_{m_1'}^{\perp})(F)}(|\det|^{1/2} \otimes \omega_{n^+-m_1',\psi_F}).
\]
Following a similar computation as in \cite[p3336]{LS13}, we have that 
\[
\sigma \otimes \omega_{n^{\prime,+},\psi_F}=\ind^{G_{n^+}\ltimes \CH_{n^+}(F)}_{P_{X_{m_1'}}\ltimes \CH(X_{m_1'}^{\perp})(F)}(\delta_{P_{X_{m_1'}}\ltimes \CH(X_{m_1'}^{\perp})}^{1/2}|\det|^{s_1'}\tau_1 \otimes (\sigma_{t-1} \otimes\omega_{n^+-m_1',\psi_F})).
\]
Here $X_{m_1'}^{\perp}$ is the complement of $X_{m_1'}$ in $V$ with respect to the $\epsilon$-hermitian form, $\CH(X_{m_1'}^{\perp})$ is the subgroup $\Res_{E/F}X_{m_1'}^{\perp}\oplus \BG_{a,F}$ of the Heisenberg group $\CH_{n^{+}}$.
We deduce that 
\[
\begin{aligned}
    m(\pi,\sigma)&=\dim \Hom_{G_{n^+}(F)}(\pi \otimes \ind^{G_{n^+}\ltimes \CH_{n^+}(F)}_{P_{X_{m_1'}}\ltimes \CH(X_{m_1'}^{\perp})(F)}(\delta_{P_{X_{m_1'}}\ltimes \CH(X_{m_1'}^{\perp})}^{1/2}|\det|^{s_1'}\tau_1 \otimes (\sigma_{t-1} \otimes \omega_{n^+-m_1',\psi_F})), \mathbbm{1}). 
\end{aligned}
\]
Using Mackey's theory, we will study the multiplicity based on the structure of the double cosets
 \[
P_{X_{m_1'}}\ltimes \CH(X_{m_1'}^{\perp})(F)\bs G_{n^+}\ltimes \CH_{n^+}(F)/ G_{n^+}(F).
 \]
 
From the computation in \cite[\S 6]{GRS11}%(see also the archimedean counterpart in \cite[Lemma 2.3.3, 2.3.6]{Ch23})
, we have the following lemma on the structure of the above double cosets.
\begin{lem}\label{lem: double cosets}
With notations as given above, the following hold. 
\begin{enumerate}
    \item  The set 
    \[
    P_{X_{m_1'}}\ltimes \CH(X_{m_1'}^{\perp})(F)\bs G_{n^+}\ltimes \CH_{n^+}(F)/ G_{n^+}(F)
    \]
    contains an open double coset $P_{X_{m_1'}}\ltimes \CH(X_{m_1'}^{\perp})(F)\gamma_{\mathrm{open}} G_{n^+}(F)$ and a closed double coset $P_{X_{m_1'}}\ltimes \CH(X_{m_1'}^{\perp})(F)\gamma_{\mathrm{closed}} G_{n^+}(F)$.
    \item One has that 
    \[
\gamma_{\mathrm{open}}^{-1}P_{X_{m_1'}}\ltimes \CH(X_{m_1'}^{\perp})\gamma_{\mathrm{open}}\cap G_{n^+}=P_{X_{m_1'}}^{(1)},
    \]
%    where $R_{m_1'-1,1}$ is the mirabolic subgroup of $\GL_{m_1'}$. %We denote this group by $S_{\mathrm{open}}$. (This is the group $P_1''$)
\item One has that 
\[
\gamma_{\mathrm{closed}}^{-1}P_{X_{m_1'}}\ltimes \CH(X_{m_1'}^{\perp})\gamma_{\mathrm{closed}}\cap G_{n^+}=P_{X_{m_1'}}.
\]
\end{enumerate}
  \end{lem}
We define that $\sigma^+:=\delta_{P_{X_{m_1'}}\ltimes X_{m_1'}^{\perp}}^{1/2}|\det|^{s_1'}\tau_1 \otimes (\sigma_{t-1} \otimes \omega_{n^+-m_1',\psi_F})$.
Then the left-hand side of the inequality can be expressed as
\begin{equation}\label{equ: multiplicity big}
m(\pi,\sigma)=\dim \Hom_{G_{n^+}(F)}(\pi \otimes \ind_{P_{X_{m_1'}}\ltimes \CH(X_{m_1'}^{\perp})(F)}^{G_{n^+}\ltimes \CH_{n^+}(F)}(\sigma^+),\mathbbm{1}).
\end{equation}
By definition, we have that 
\[
\ind_{P_{X_{m_1'}}\ltimes \CH(X_{m_1'}^{\perp})(F)}^{G_{n^+}\ltimes \CH_{n^+}(F)}(\sigma^+)=\Gamma^{\CC}(P_{X_{m_1'}}\ltimes \CH(X_{m_1'}^{\perp})(F)\bs G_{n^+}\ltimes \CH_{n^+}(F),\CE_{\sigma^+}),
\]
that is, the space of compact-supported sections on the bundle $\CE_{\sigma^+}$, where 
\[
\CE_{\sigma^+} := P_{X_{m_1'}}\ltimes \CH(X_{m_1'}^{\perp})(F)\bs ( (G_{n^+}\ltimes \CH_{n^+}(F))\times \sigma^+).
\]
Here the left $P_{X_{m_1'}}\ltimes \CH(X_{m_1'}^{\perp})(F)$-action on $(G_{n^+}\ltimes \CH_{n^+}(F))\times \sigma^+$ is given by $p.(g,v)=(pg,\sigma^+(p)v)$. We set 
\[
\begin{aligned}
    & \CX=P_{X_{m_1'}}\ltimes \CH(X_{m_1'}^{\perp})(F)\bs G_{n^+}\ltimes \CH_{n^+}(F),\\
& \CU=P_{X_{m_1'}}\ltimes \CH(X_{m_1'}^{\perp})(F)\bs P_{X_{m_1'}}\ltimes \CH(X_{m_1'}^{\perp})(F)\gamma_{\mathrm{open}}G_{n^+}(F), \\
& \CZ=\CX-\CU=P_{X_{m_1'}}\ltimes \CH(X_{m_1'}^{\perp})(F)\bs P_{X_{m_1'}}\ltimes \CH(X_{m_1'}^{\perp})(F)\gamma_{\mathrm{closed}}G_{n^+}(F)
\end{aligned}
\]
and obtain an exact sequence
\[
0\to \Gamma^{\CC}(\CU,\CE_{\sigma^+})\to \Gamma^{\CC}(\CX,\CE_{\sigma^+})\to \Gamma^{\CC}(\CZ,\CE_{\sigma^+})\to 0,
\]
%which implies an exact sequence
%\[
%0\to \pi\wh{\otimes}\Gamma^{\CC}(\CU,\CE_{\sigma^+})\to \pi\wh{\otimes}\Gamma^{\CC}(\CX,\CE_{\sigma^+})\to \pi\wh{\otimes}\Gamma^{\CC}(\CZ,\CE_{\sigma^+})\to 0,
%\]
which implies an exact sequence 
\begin{equation}\label{equ: exact 1}
\begin{aligned}
0 & \to \Hom_{G_{n^+}(F)}(\pi \otimes \Gamma^{\CC}(\CZ,\CE_{\sigma^+}), \mathbbm{1})  \to \Hom_{G_{n^+}(F)}(\pi \otimes \Gamma^{\CC}(\CX,\CE_{\sigma^+}), \mathbbm{1}) \\
& \to\Hom_{G_{n^+}(F)}(\pi \otimes \Gamma^{\CC}(\CU,\CE_{\sigma^+}), \mathbbm{1}).
\end{aligned}
\end{equation}
Moreover, from Lemma \ref{lem: double cosets}, we have
\begin{equation}\label{equ: open double coset}
\Gamma^{\CC}(\CU,\CE_{\sigma^+})=\ind_{P_{X_{m_1'}}^{(1)}}^{G_{n^+}}(\sigma^+|_{P_{X_{m_1'}}^{(1)}})=\ind_{P_{X_{m_1'}}^{(1)}}^{G_{n^+}}(|\det|^{s_1'+s_0}\tau_1|_{\GL_{m_1'}^{(1)}} \otimes \sigma_{t'-1})
\end{equation}
Here $s_0$ is the constant satisfying $\delta_{P_{X_{m_1'}}\ltimes \CH(X_{m_1'}^{\perp})}^{1/2}=|\det|^{s_0}$. %and $\GL_{m}^{(i)}$ is defined as in \eqref{equ: derivative parabolic groups}. In particular, $\GL_{m_1'}^{(1)}$ is the mirabolic subgroup and $\GL^{(m_1')}_{m_1'}$ is the maximal unipotent subgroup.

For $1\leq i\leq m_1'$, we set $\mu_{i}$ be the character on the unipotent radical of $\GL_{X_{m_1'}}^{(i)}(F)$ and  $P_{X_{m_1'}}^{(i)}(F)$ obtain from the restriction of $\psi_{X_{m_1'}}$defined in (\ref{psix}). The right side of the inequality in Proposition \ref{prop: three ingredients}(2) is
\[
\begin{aligned}
m(\pi,\sigma_{t'-1})=\dim \Hom_{H_{\CO_{m_1'}}(F)}(\CJ_{\CO_{m_1'}}(\pi) \otimes\sigma_{t'-1}, \mathbbm{1})
\end{aligned}
\]
where the $\Hom$-space on the right-hand side is equal to 
\[
\Hom_{G_{n^+-m_1'}(F)}(\delta_{P_{X_{m_1'}}^{(m_1')}}\otimes V_{\pi}/\Span\{\pi(u)v-\mu_{m_1'}(u)v\mid u\in N_{X_{m_1'}}^{(m_1')}(F),v\in V_{\pi}\}\otimes \sigma_{t'-1} \otimes \omega_{n^+-m_1',\psi_F}, \mathbbm{1}),
\]
which can be written as 
\[
\Hom_{P_{X_{m_1'}}^{(m_1')}(F)}(\delta_{P_{X_{m_1'}}^{(m_1')}}\otimes\pi \otimes \sigma_{t'-1} \otimes \omega_{n^+-m_1',\psi_F},\mu_{m_1'}).
\]
Hence we obtain that 
\begin{align}\label{equ: rhs expansion}
m(\pi,\sigma_{t'-1})=\dim \Hom_{G_{n^+}(F)}(\pi \otimes \Gamma^{\CC}(\CU,\CE_{\sigma^+})^{(m_1')},\mathbbm{1}),
\end{align}
where
\[
\Gamma^{\CC}(\CU,\CE_{\sigma^+})^{(m_1')}=\ind_{P_{X_{m_1'}}^{(m_1')}(F)}^{G_{n^+}(F)}(\mu_{m_1'}^{-1}\otimes \sigma_{t'-1} \otimes \omega_{n^+,\psi_F}|_{G_{n-m_1'}}).
\]
We set
\[\Gamma_{\mathrm{open}}=\Gamma^{\CC}(\CU,\CE_{\sigma^+})/
\Gamma^{\CC}(\CU,\CE_{\sigma^+})^{(m_1')}.\]
Then by \eqref{equ: open double coset}, we have that 
\begin{equation}
\Gamma_{\mathrm{open}}=\ind_{P_{X_{m_1'}}^{(1)}(F)}^{G_{n^+}(F)}((|\det|^{s_1'+s_0}\tau_1|_{\GL_{m_1'}^{(1)}}/\ind_{\GL_{m_1'}^{(m_1')}}^{\GL_{m_1'}^{(1)}}(\mu_{m_1'}^{-1}))\otimes\sigma_{t'-1}).
\end{equation}
Therefore, we obtain an exact sequence 
\begin{equation}\label{equ: open contribution}
\begin{aligned}
    0  \to \Hom_{G_{n^+}(F)}(\pi \otimes \Gamma_{\mathrm{open}},\mathbbm{1})&\to \Hom_{G_{n^+}(F)}(\pi \otimes\Gamma^{\CC}(\CU,\CE_{\sigma^+}),\mathbbm{1}) \\
    & \to\Hom_{G_{n^+}(F)}(\pi \otimes \Gamma^{\CC}(\CU,\CE_{\sigma^+})^{(m_1')},\mathbbm{1}).
    \end{aligned}
\end{equation}

The structure of $|\det|^{s_1'}\tau_1|_{\GL_{m_1'}^{(1)}}/\ind_{\GL_{m_1'}^{(m_1')}(F)}^{\GL_{m_1'}^{(1)}(F)}(\mu_{m_1'}^{-1})$ can be computed with the derivative theory in \cite{BZ77}. I n particular, we recall the following result from  \cite[\S 4.3]{BZ77}.
\begin{lem}
For an irreducible admissible representation $\tau$ of $\GL_n(F)$, there is a filtration
\[
1=\tau^{0}\subset \tau^{(1)}\subset \cdots \subset \tau^{(n-1)}=\tau|_{\GL_{n}^{(1)}(F)}
\]
such that
\[
\tau^{(i+1)}/\tau^{(i)}\cong \ind_{\GL_n^{(k)}(F)}^{\GL_n^{(1)}(F)}(\Delta^k\tau\otimes \mu_k^{-1})
\]
where $\Delta^k$ denotes the $k$-th Bernstein-Zelevinsky derivative. 
\end{lem}
Therefore, we obtain a filtration in 
$\sigma_{\mathrm{open}}$ with graded pieces
\begin{equation}\label{equ: graded pieces}
    \ind_{P_{X_{m_1'}}^{(i)}(F)}^{G_{n^+}(F)}(|\det|^{s_1'+s_0}\Delta^k\tau_1 \otimes \sigma_{t-1}\otimes\mu_k^{-1}),\quad k=1,\cdots,m_1'.
\end{equation}

\begin{lem}[Vanishing results]\label{lem: vanishing}
%\begin{enumerate}
Let $\sigma\in\Pi_F(\Res_{E/F}\GL_{r})$ be a tempered representation and $\pi_1$ be a smooth representation of $G_{n^+-r}(F)$. Let $\pi_2=\RI(\udl{s},\tau_1,\cdots,\tau_t,\pi_0)$ be a representation of $G_{n^+}(F)$ where $\udl{s}=(s_1,\cdots,s_t)$ for $s_1\geq s_2\geq \cdots \geq s_t\geq 0$. Then
 \[
\Hom_{G_{n^+}(F)}(\RI(s',\sigma, \pi_1) \otimes \pi_2, \mathbbm{1})=0
\]
when $s'>s_1$.
\end{lem}

\begin{proof}
Our proof follows from that in \cite[p177]{MW12}. From the second adjointness theorem, we have
\[
\Hom_{G_{n^+}(F)}(\RI(s',\sigma, \pi_1),\pi_2^{\vee})=\Hom_{M_{X_{m_1}}(F)}(|\det|^{s'}\sigma \otimes \pi_1,\Jac_{P^-}(\pi_2^{\vee})),
\]
where $M_{X_{m_1}} = \Res_{E/F}\GL_{m_1} \times G_{n^+-m_1}$ and $P^-$ is the opposite parabolic subgroup with Levi subgroup $M_{X_{m_1}}$.
Suppose that $\Hom_{G_{n^+}(F)}(\RI(s',\sigma, \pi_1),\pi_2^{\vee})\neq 0$. Since $\Jac_{P^-}(\pi_2^{\vee})$ is admissible of finite length, there is an irreducible admissible representation $\pi_1'$ such that
\[
\Hom_{M_{X_{m_1}}(F)}(|\det|^{s'}\sigma \otimes \pi_1',\Jac_{P^-}(\pi_2^{\vee}))\neq 0,
\]
which leads to a contradiction when comparing the exponent (\cite[Lemma A.0.4]{Ch23}).  
\end{proof}

\begin{proof}[Proof for Part (2) of Proposition \ref{prop: three ingredients}] First, we take both sides of the inequality as the dimensions studied in (\ref{equ: multiplicity big}) and (\ref{equ: rhs expansion}). Based on the structures computed in
(\ref{equ: exact 1}), (\ref{equ: open contribution}) and (\ref{equ: graded pieces}), it suffices to show that
\begin{enumerate}
    \item $\Hom_{G_{n^+}(F)}(\pi \otimes \Gamma^{\CC}(\CZ,\CE_{\sigma^+}),\mathbbm{1})=0$, when  $s_1'\geq s_1$, and
    \item $\Hom_{G_{n^+}(F)}(\pi \otimes \ind_{P_{X_{m_1'}}^{(i)}(F)}^{G_{n^+}(F)}(|\det|^{s_1'+s_0}\Delta^i\tau_1 \otimes\sigma_{t-1}\otimes \mu_i^{-1}),\mathbbm{1})=0$ for all $1\leq i\leq m_1'$, when  $s_1'\geq s_1$.
\end{enumerate}

On the one hand,  for  $g\in P_{X_{m_1'}}(F)$, $x_{m_1'}\in X_{m_1'}$,\[\delta_{P_{X_{m_1'}}\ltimes X_{m_1'}}(g\ltimes x_{m_1'})=|\det(g_{\GL(X_{m_1'})})|\delta_{P_{X_{m_1'}}}(g)\]
Here $g_{\GL(X_{m_1'})}$ is the $\GL(X_{m_1'})$-factor of $g$ in the Levi component. Then we have
\[
\begin{aligned}
\Gamma^{\CC}(\CZ,\CE_{\sigma^+})=\ind_{P_{X_{m_1'}}(F)}^{G_{n^+}(F)}(\sigma^+|_{ P_{X_{m_1'}}})=\RI(s_1'+\frac{1}{2},\tau_1,\sigma_{t-1} \otimes \omega_{n^+-m_1',\psi_F}|_{G_{n-m_1'}}).
\end{aligned}
\]
Then from Lemma \ref{lem: vanishing}, we have
\begin{equation}\label{equ: vanishing closed}
\Hom_{G_{n^+}(F)}(\pi \otimes \Gamma^{\CC}(\CZ,\CE_{\sigma^+}),\mathbbm{1})=0,\quad \text{when } s_1'\geq s_1.
\end{equation}
On the other hand, for $1\leq i\leq m_1'$, 
\[
\begin{aligned}   
&\Hom_{G_{n^+}(F)}(\pi \otimes \ind_{P_{X_{m_1'}}^{(i)}(F)}^{G_{n^+}(F)}(|\det|^{s_1'+s_0}\Delta^i\tau_1 \otimes\sigma_{t-1}\otimes\mu_i^{-1}),\mathbbm{1})\\
=\,&\Hom_{G_{n^+}(F)}(\pi \otimes \RI(s_1'+\frac{i}{2},\Delta^i\tau_1, \ind_{G_{n^+-m_1'+i}\cap P_{X_{m_1'}}^{(i)}(F)}^{G_{n^+-m_1+i}(F)}(\sigma_{t'-1} \otimes \omega_{n^+-m_1',\psi_F}\otimes \mu_i^{-1})),\mathbbm{1}).
\end{aligned}
\]
Then from Lemma \ref{lem: vanishing}, we have
\[
\Hom_{G_{n^+}(F)}(\pi \otimes \ind_{P_{X_{m_1'}}^{(i)}(F)}^{G_{n^+}(F)}(|\det|^{s_1'+s_0}\Delta^i\tau_1 \otimes \sigma_{t-1}\otimes \mu_i^{-1}),\mathbbm{1})=0,\quad \text{when } s_1'\geq s_1.
\]

This completes the proof for Part (2) of Proposition \ref{prop: three ingredients}.
\end{proof}

\subsubsection{Mathematical induction}
We are ready to prove the first inequality
\begin{equation}\label{equ: first inequality}
m(\RI(\udl{s},\rho_1,\dots\rho_t,\pi_0),\RI(\udl{s}',\tau_1,\dots\tau_{t'},\sigma_0))\leq m(\pi_0,\sigma_0)   
\end{equation}
in Proposition \ref{pro: multiplicity} 
using Parts (1) and (2) of Proposition \ref{prop: three ingredients}. 

First, we reduce it to the equal-rank cases, that is, the situations with $n^+=n^{\prime,+}$. When $n^+\neq n^{\prime,+}$, say $n^+>n^{\prime,+}$, we may choose a supercuspidal representation $\rho_0$ of $\GL_{n^+-n^{\prime,+}}(E)$ such that by Part (1) of Proposition \ref{prop: three ingredients}, we have 
\[
m(\RI(\udl{s},\rho_1,\dots\rho_t,\pi_0),\RI(\udl{s}',\tau_1,\dots\tau_{t'},\sigma_0))=m(\RI(\udl{s}^+,\rho_0,\rho_1,\dots\rho_t,\pi_0),\RI(\udl{s}',\tau_1,\dots\tau_{t'},\sigma_0))
\]
with $\udl{s}^+=(0,s_1,\cdots,s_{t})$. The comparison of $m(\RI(\udl{s}^+,\rho_0,\rho_1,\dots\rho_t,\pi_0),\RI(\udl{s}',\tau_1,\dots\tau_{t'},\sigma_0))$ with $m(\pi_0,\sigma_0)$ is an equal-rank case, which can be proved by induction on $N= N(\RI(\udl{s},\rho_1,\dots\rho_t,\pi_0),\RI(\udl{s}',\tau_1,\dots\tau_{t'},\sigma_0))$, the number of nonzero elements among $s_1,\cdots,s_t,s_1',\cdots,s_{t'}'$.
%\[N := N(\RI(\udl{s},\rho_1,\dots\rho_t,\pi_0),\RI(\udl{s}',\tau_1,\dots\tau_{t'},\sigma_0))=\sum_{s_i\neq 0}1+\sum_{s_i'\neq 0}1.\] 

When $N(\RI(\udl{s},\rho_1,\dots\rho_t,\pi_0),\RI(\udl{s}',\tau_1,\dots\tau_{t'},\sigma_0))=0$, the $s_i$ and $s_i'$ are all equal to zero. By the full decomposition of $\pi=\RI(\udl{s},\rho_1,\dots\rho_t,\pi_0)$ and $\sigma=\RI(\udl{s}',\tau_1,\dots\tau_{t'},\sigma_0)$ computed in \cite{Ar13} and \cite{Mok15}  (see \cite[Desideratum 2.1(4)]{At16}) and then apply the local Gan-Gross-Prasad conjecture for tempered $L$-parameters for each component, we have $m(\pi_0,\sigma_0)=1$ if and only if there exactly one pair of $(\pi',\sigma')$ such that  $m(\pi',\sigma')=1$, where  $\pi',\sigma'$ are irreducible components of $\pi,\sigma$.

When $N(\RI(\udl{s},\rho_1,\dots\rho_t,\pi_0),\RI(\udl{s}',\tau_1,\dots\tau_{t'},\sigma_0))=k\geq 1$, suppose that the inequality holds for all $N=k-1$ situations. We may assume that
\[
s_1\geq \cdots \geq s_t\geq 0,\quad \text{and}\quad s_1'\geq \cdots \geq s_{t'}'\geq 0.
\]

\begin{enumerate}
    \item If $s_1'\geq s_1$, then from Part (2) of Proposition \ref{prop: three ingredients}, we have 
    \[
    \begin{aligned}
  & m(\RI(\udl{s},\rho_1,\dots\rho_t,\pi_0),\RI(\udl{s}',\tau_1,\dots\tau_{t'},\sigma_0)) \\
 \leq \, & m(\RI(\udl{s},\rho_1,\dots\rho_t,\pi_0),\RI((s_2',\cdots,s_{t'}'),\tau_2,\dots\tau_{t'},\sigma_0)).
\end{aligned}
    \]
By Part (1) of Proposition \ref{prop: three ingredients}, we can choose a supercuspidal representation $\tau_0$ of $\GL_{m_1'}(E)$ such that
\[    
\begin{aligned} 
& m(\RI(\udl{s},\rho_1,\dots\rho_t,\pi_0),\RI((0,s_2',\cdots,s_{t'}'),\tau_0,\tau_2,\dots\tau_{t'},\sigma_0)) \\
=\, & m(\RI(\udl{s},\rho_1,\dots\rho_t,\pi_0),\RI((s_2',\cdots,s_{t'}'),\tau_2,\dots\tau_{t'},\sigma_0)),
\end{aligned}
\]
noting that 
\[
N(\RI(\udl{s},\rho_1,\dots\rho_t,\pi_0),\RI((0,s_2',\cdots,s_{t'}'),\tau_0,\tau_2,\dots\tau_{t'},\sigma_0))=k-1. 
\]
\item If $s_1'<s_1$,
then from Part (2) of Proposition \ref{prop: three ingredients}, we have 
    \[
    \begin{aligned}  
 & m(\RI(\udl{s},\rho_1,\dots\rho_t,\pi_0),\RI(\udl{s}',\tau_1,\dots\tau_{t'},\sigma_0)) \\
=\, &m(\RI(\udl{s}',\tau_1,\dots\tau_{t'},\sigma_0),\RI(\udl{s},\rho_1,\dots\rho_t,\pi_0))\\
\leq\, &m(\RI(\udl{s}',\tau_1,\dots,\tau_{t'},\sigma_0),\RI((s_2,\cdots,s_{t}),\rho_2,\dots\rho_t,\pi_0)).
    \end{aligned}
    \]
By Part (1) of Proposition \ref{prop: three ingredients}, we can choose a supercuspidal representation $\tau_0$ of $\GL_{n_1}(E)$ such that
\[   
\begin{aligned}
 & m(\RI(\udl{s}',\tau_1,\dots,\tau_{t'},\sigma_0),\RI((s_2,\cdots,s_{t}),\rho_1,\dots\rho_t,\pi_0)) \\
=\, &m(\RI(\udl{s}',\tau_1,\tau_2\dots\tau_{t'},\pi_0),\RI((0,s_2,\cdots,s_{t}),\rho_0,\rho_2,\dots\rho_{t},\sigma_0))\\
=\, & m(\RI((0,s_2,\cdots,s_{t}),\rho_0,\rho_2,\dots\rho_{t},\pi_0),\RI(\udl{s}',\tau_1,\dots\tau_t,\sigma_0)),
\end{aligned}
\]
noting that
\[
N(\RI((0,s_2,\cdots,s_{t}),\rho_0,\rho_2,\dots\rho_{t},\pi_0),\RI(\udl{s}',\tau_1,\dots\tau_t,\sigma_0))=k-1.
\]
\end{enumerate}
Therefore, in both situations, we can reduce (\ref{equ: first inequality}) to the inequality for 
\[
m(\RI(\udl{s},\rho_1,\dots\rho_t,\pi_0),\RI((0,s_2',\cdots,s_{t'}'),\tau_0,\tau_2,\dots\tau_{t'},\sigma_0))
\]
or
\[
m(\RI((0,s_2,\cdots,s_{t}),\rho_0,\rho_2,\dots\rho_{t},\pi_0),\RI(\udl{s}',\tau_1,\dots\tau_t,\sigma_0)),
\]  
which can be obtained by the induction hypothesis. This completes the mathematical induction.

\subsection{The second inequality}
In this section, we prove the basic forms of second inequality (Part (3) of Proposition \ref{prop: three ingredients}) following the integral method in \cite{MW12}.
\subsubsection{Construction of the integral}
We work in the situation when $n^+=n^{\prime,+}$, under the assumption that $m(\pi_0,\sigma_0)\neq 1$. We will prove by construction that 
\[
m(\RI(\udl{s},\rho_1,\dots\rho_t,\pi_0),\RI(\udl{s}',\tau_1,\dots,\tau_{t'},\sigma_0))\neq 0,
\]

We first define $\pi_{\udl{s}} (\udl{s}\in \BC^{t})$ using induction from $P$ in \eqref{SM1} when $\RG_{n^+}$ is not metaplectic
and in \eqref{SM-Mp} when $\RG_{n^+}$ is metaplectic. We define $\sigma_{\udl{s}'} (\udl{s}'\in \BC^{t'})$ similarly using induction from $P'$. Then
\[
\pi_{\udl{s}}=\RI(\udl{s},\rho_1,\dots\rho_t,\pi_0),\quad \sigma_{\udl{s}'}=\RI(\udl{s}',\tau_1,\dots,\tau_{t'},\sigma_0),\quad \udl{s}\in \BR^{t},\udl{s}'\in \BR^{t'}
\]
Let $K$ and $K'$ be maximal compact subgroups of $G_{n^+}(F)$ and $H_{n^+}(F)$, which are in good position with respect to $P$ and $P'$ respectively. Then the underlying spaces of $\pi_{\udl{s}}$ and $\sigma_{\udl{s'}}$ can be realized as
\[
\pi_K=\Ind_{P\cap K}^{K}(\rho_1 \otimes \cdots \otimes \rho_{t} \otimes \pi_0)\quad \text{and}\quad
\sigma_{K'}=\Ind_{P'\cap K'}^{K'}(\tau_1 \otimes \cdots \otimes \tau_{t'} \otimes \sigma_0)
\]
respectively, and the actions of $\pi_{\udl{s}}$ and $\sigma_{\udl{s}'}$ are given by 
\[
\pi_{\udl{s}}(pk)v=\left(\delta_{P}^{1/2}|\det|^{s_1}\rho_1 \otimes \cdots \otimes |\det|^{s_t}\rho_t \otimes \pi_0\right)(p)\pi_K(k)v, \text{ 
for } p\in P, k\in K,
\]
\[
\sigma_{\udl{s}'}(p'k')v'=\left(\delta_{P'}^{1/2}|\det|^{s_1'}\tau_1 \otimes \cdots \otimes |\det|^{s_{t'}'}\tau_{t'} \otimes \sigma_0\right)(p')\sigma_{K'}(k')v', \text{ 
for } p'\in P', k'\in K'.
\]
Following \cite{MW12}, we define the integral
\[
\begin{aligned}
\RI_{v,v^*,v'\otimes w,v^{\prime,*}\otimes w^*}(\udl{s},\udl{s}')&=\int_{G_{n^+}(F)}\langle\pi_{\udl{s}}(g)v, v^*\rangle \langle(\sigma_{\udl{s'}}\otimes \omega_{n^+, \psi_F})(g,0)v'\otimes w,v^{\prime,*}\otimes w^*\rangle dg\\
&=\int_{G_{n^+}(F)}\langle\pi_{\udl{s}}(g)v, v^*\rangle \langle\sigma_{\udl{s'}}(g)v',v^{\prime,*}\rangle\langle \omega_{n^+,\psi_F}(g,0)w, w^*\rangle dg,
\end{aligned}
\]
where $v\otimes v' \otimes w\in \pi_K\otimes \sigma_{K'}\otimes \omega_{n^+, \psi_F}$ and 
$ v^*\otimes v^{\prime, *} \otimes w^* \in \pi_K^\vee\otimes \sigma_{K'}^\vee \otimes \omega_{n^+,\psi_F}^\vee$. 

\begin{prop}\label{pro: integral}
When $m(\pi_0,\sigma_0)\neq 0$, the following hold.

\begin{enumerate}
\item 
There is a constant $c>0$ such that the integral $\RI_{v,v^*,v'\otimes w,v^{\prime,*}\otimes w^*}(\udl{s},\udl{s}')$ is absolutely convergent when $\Re(s_i)<c$;

\item There exists $(\udl{s},\udl{s}')\in (i\BR)^{t+t'}$ such that
\[
\RI_{v,v^*,v'\otimes w,v^{\prime,*}\otimes w^*}(\udl{s},\udl{s}')\neq 0;
\]
    \item 
The integral $\RI_{v,v^*,v'\otimes w,v^{\prime,*}\otimes w^*}(\udl{s},\udl{s}')$ has a meromorphic continuation to $\BC^{r}\times \BC^{r'}$.
\end{enumerate}
\end{prop}

We can prove the second inequality 
\begin{equation}
m(\pi_{\udl{s}_0},\sigma_{\udl{s}'_0})\geq m(\pi_0,\sigma_0)
\end{equation}
using this proposition. Since $\pi_0$ and $\sigma_0$ are irreducible, we have $m(\pi_0,\sigma_0)\leq 1$ by the multiplicity-one theorem (\cite{AGRS10, Su12}).
 When $m(\pi_0,\sigma_0)=0$,  the inequality always holds. When $m(\pi_0,\sigma_0)=1$, from Proposition \ref{pro: integral}, we obtain a nonzero meromorphic family in \[\Hom_{G_{n^+}(F)}(\pi_{\udl{s}} \otimes \sigma_{\udl{s}'} \otimes \omega_{n^+,\psi_F}, \mathbbm{1}).\]
 By taking the principal term at $\udl{s}=\udl{s_0}$, $\udl{s}'=\udl{s}_0'$, we obtain a nonzero element in 
 \[
 \Hom_{G_{n^+}(F)}(\pi_{\udl{s}_0} \otimes \sigma_{\udl{s}'_0} \otimes \omega_{n^+,\psi_F},\mathbbm{1}).
 \]
 Therefore,
 \[ m(\pi_{\udl{s}_0},\sigma_{\udl{s}'_0})\geq 1=m(\pi_0,\sigma_0).
 \]
This completes the proof of Proposition \ref{pro: multiplicity}.

\subsubsection{Estimates and Tempered intertwinings}
We use some estimates to prove Proposition \ref{pro: integral}. For two functions $f_1$ and $f_2$ on $G$, we denote by $f_1\ll f_2$ if there is a constant $C>0$ such that  
\[
| f_1(g) | <C | f_2(g) |,\quad \text{ for all }g\in G.
\]
%We denote by $f_1\sim f_2$ if $f_1\ll f_2$ and $f_2\ll f_1$.

For a reductive algebraic group $G$ over $F$, we denote by $\Xi^G(g)$ the Harish-Chandra $\Xi$-function of $G(F)$, and denote by $\iota_G(g)$  the log-norm on $G(F)$.

%We let $A^+$ be a maximal $F$-split torus of $G_{n^+}(F)$ preserving a polarization $V = X \oplus V_0 \oplus X^\vee$, with $V_0$  the anisotropic kernel and $X$  maximal totally isotropic.

When $G$ is unitary or symplectic, we fix a minimal parabolic subgroup $P_0=M_0\ltimes N$  of $G_{n^+}$ and let $A_0$ be the maximal split subtorus of $M_0$. Let $r$ be the split rank of $A_0$. The Cartan decomposition of $G_{n^+}$ gives
 \[
 G_{n^+}(F)=K_{G_{n^+}}A^+_{G_{n^+}}K_{G_{n^+}}, \quad A^+_{G_{n^+}}=\set{a\in A_0(F) | |\alpha(a)|\geq 1,\alpha\in R(A_0,P_0)},
 \]
 where $R(A_0,P_0)$ is the set of roots of $A_0$ in the unipotent radical of $P_0$. When taking $P_0$ to be upper-triangular, we have
 \begin{equation}\label{equ: Cartan parameterize 1}
A^+_{G_{n^+}}=\set{ \diag(\varpi^{n_1}_E,\cdots,\varpi^{n_r}_E) | n_1\geq n_2\geq \cdots n_r\geq 0},
 \end{equation}
 where $\varpi_E$ is the uniformizer of $\Fo_E$.

When $G$ is metaplectic, if we take $K_{G_{n^+}}$, $A^+_{G_{n^+}}$ to be the inverse images of $K_{\Sp(2n^+)}$, $A^+_{\Sp(2n^+)}$, then
\[
G_{n^+}(F)=K_{G_{n^+}}A^+_{G_{n^+}}K_{G_{n^+}}.
\]
From now on, we lift representations of $G_{n^+}(F)$ and $H_{n^+}(F)$ to representations of $\wt{G}_{n^+}$ (\eqref{equ: define wtG}). For preimages $\wt{a}_1$ and $\wt{a}_2$ of $a\in A_{\Sp(2n^+)}^+$, 
\[
K_{G_{n^+}}\wt{a}_1K_{G_{n^+}}=K_{G_{n^+}}\wt{a}_2K_{G_{n^+}}.
\]
We write this space of $K_{G_{n^+}}a K_{G_{n^+}}$ for simplicity.

\begin{lem}\label{lem: estimate of matrix coefficient}
Given $v\otimes v' \otimes w\in \pi_K\otimes \sigma_{K'}\otimes \omega_{n^+,\psi_F}$ and 
$ v^*\otimes v^{\prime, *} \otimes w^* \in \pi_K^\vee\otimes \sigma_{K'}^\vee \otimes \omega_{n^+,\psi_F}^\vee$, there exists % an integer $d>0$ and 
$\epsilon>0$  such that  % for every  $d>0$
\[
\langle\pi_{\udl{0}}(g)v, v^*\rangle \langle(\sigma_{\udl{0}'}\otimes \omega_{n^+, \psi_F})(g,0)v'\otimes w,v^{\prime,*}\otimes w^*\rangle \ll \Xi^{G_{n^+}}(g)^2
%\iota_G(g)^{-d}
\exp(-\epsilon\iota_{G_{n^+}}(g)).
\]
\end{lem}

\begin{proof}
We take  $\wt{g}$  to be  a preimage of $g$ in  $\wt{G}_{n^+}(F)$.   Since $\pi_{\udl{0}}$ and $\sigma_{\udl{0}'}$ are tempered and the matrix coefficients of tempered representations are weakly Harish-Chandra Schwartz (\cite[III.2]{W03}, \cite[Proposition 2.3.1]{Li12}), there exists $d_1,d_2>0$ such that
   \[
\langle\pi_{\udl{0}}(\wt{g})v, v^*\rangle \ll \Xi^{\wt{G}_{n^+}}(\wt{g})\iota_{\wt{G}_{n^+}}(\wt{g})^{d_1}
    \]
\[
\langle\sigma_{\udl{0}'}(\wt{g})v', v^{\prime,*}\rangle \ll \Xi^{\wt{G}_{n^+}}(\wt{g})\iota_{\wt{G}_{n^+}}(\wt{g})^{d_2}.
\]
It is known and easy to show that (cf. proof of \cite[Lemma 9.1]{GI11}),
   \[    
   \langle\omega_{n^+,\psi_F}(\wt{a}) w,w^*\rangle\ll |a_1\cdots a_r|^{1/2},
   %\iota_{G_{n^+}}(\wt{a})^{-d'},
   \quad a =\diag(a_1,\cdots ,a_r)\in A^+_{G_{n^+}},
   \]
   where we fix a splitting of the metaplectic cover over $A_0(F)$.
 To prove this estimation, we take a polarization of the underlying symplectic or skew-Hermitian space $V = X \oplus V_0 \oplus X^\vee$ which is preserved by $A_0(F)$, with $V_0$ the anisotropic kernel and $X\cong E^r$ maximal totally isotropic. The Weil representation $\omega_{n^+,\psi_F}$ can be realized as the mixed model $\mathcal{S}(X)\otimes \omega_{V_0, \psi_F}$, where $\mathcal{S}(X)$ is the space of Bruhat-Schwartz functions on $X$ and $\omega_{V_0,\psi_F}$ is the Weil representation associated to $V_0$. Then $a =(a_1,\ldots,a_r)\in A_0(F)$ acts on $w= \phi\otimes\phi_0\in \mathcal{S}(X)\otimes \omega_{V_0, \psi_F}$ by 
   \[
   \omega_{n^+,\psi_F}(\wt{a})w =  \pm|a_1\cdots a_r|^{1/2}\phi(\cdot\, a) \otimes \phi_0.
   \]
For Bruhat-Schwartz functions $\phi, \phi'\in \mathcal{S}(X)$, one has that 
\[
\int_{X}\phi(xa) \overline{\phi'(x)}dx \leq C_{\phi, \phi'},\quad \wt{a}\in A^+_{\wt{G}_{n^+}}
\]
for some constant $C_{\phi,\phi'}$. 
This easily implies the required estimation 
\[
\langle\omega_{n^+,\psi_F}(\wt{a}) w,w^*\rangle\ll |a_1\cdots a_r|^{1/2},\quad a \in A^+_{\wt{G}_{n^+}}.
\]
From the Cartan decomposition of $\wt{G}_{n^+}$,  there exists $\epsilon>0$ such that
\[
\langle\omega_{n^+,\psi_F}(\wt{g}) w,w^*\rangle\ll 
%\iota_{G_{n^+}}(\wt{g})^{-d'}
\exp(-\epsilon \iota_{G_{n^+}}(\wt{g})).
\]
The lemma follows easily. 
\end{proof}

From \cite[Lemme II.1.5]{W03} and \cite[Proposition 1.5.1 (v)]{BP20}, there exists $d>0$ such that the integral
\[
\int_{G_{n^+}(F)}\Xi^{G_{n^+}}(g)^2\iota_{G_{n^+}}(g)^{-d}dg
\]
is absolutely convergent. Using this and Lemma \ref{lem: estimate of matrix coefficient}, we obtain Proposition \ref{pro: integral} (1).

For $\udl{s}\in (i\BR)^t$ and $\udl{s}'\in (i\BR)^{t'}$, the tempered intertwining $\CL_{\pi}$ is defined as 
\[
\CL_{\pi_{\underline{s}},\sigma_{\underline{s}'}}(v\otimes v'\otimes w,v^*\otimes v^{\prime,*}, w^*)=\RI_{v,v^*,v'\otimes w,v^{\prime,*}\otimes w^*}(\udl{s},\udl{s}').
\]
The following theorem is crucial for many approaches to the local Gan-Gross-Prasad conjecture for tempered $L$-parameters. This result has been established in all the other situations. We refer to \cite[Proposition 5.7]{W12a} for non-archimedean Bessel cases, to \cite[Proposition 7.2.1]{BP20} for archimedean Bessel cases and to \cite[Theorem 3.2]{X3} for archimedean Fourier-Jacobi cases.
\begin{thm}\label{thm: tempered intertwining nonzero}
For tempered representations $\pi_V,\pi_W$ of $G_{n^+},H_{n^+}$, if $m(\pi_V,\pi_W)\neq 0$, then $\CL_{\pi_V\wh{\otimes} \pi_W}\neq 0$.
\end{thm}
\begin{proof}
The proof follows from the arguments in \cite[\S 3]{X3}  verbatim based on the estimate in Lemma \ref{lem: estimate of matrix coefficient}.  
\end{proof}

\begin{proof}[Proof for (2) of Proposition \ref{pro: integral}]
 We choose $\udl{s}\in (i\BR)^t$, $\udl{s}'\in (i\BR)^{t'}$ such that $s_i+s_j\neq 0$ for all $1\leq i\leq j\leq t$ and $s_i'+s_j'\neq 0$ for all $1\leq i\leq j \leq t'$, then $\pi_{\udl{s}},\sigma_{\udl{s'}}$ are irreducible and have the same $\chi$-parameter as $\pi_0,\sigma_0$, respectively. From the local Gan-Gross-Prasad conjecture for tempered parameters,
$m(\pi_0,\sigma_0)\neq 0$ implies 
$m(\pi_{\udl{s}},\sigma_{\udl{s'}})\neq 0$. From Theorem \ref{thm: tempered intertwining nonzero}, we have
$\CL_{\pi_{\underline{s}},\sigma_{\underline{s}'}}\neq 0$, so there exist $v\otimes v' \otimes w\in \pi_K\otimes \sigma_{K'}\otimes \omega_{n^+,\psi_F}$ and 
$ v^*\otimes v^{\prime, *} \otimes w^* \in \pi_K^\vee\otimes \sigma_{K'}^\vee \otimes \omega_{n^+,\psi_F}^\vee$ such that
\[
\RI_{v,v^*,v'\otimes w,v^{\prime,*}\otimes w^*}(\udl{s},\udl{s}')\neq 0.
\]
\end{proof}

\subsubsection{Casselman's canonical pairing}

One key ingredient of M\oe glin and Waldspurger's proof is Casselman's canonical pair for irreducible admissible representations. Since we are treating the Fourier-Jacobi models, we need the corresponding results for Heisenberg-oscillator representations.

Let $\wt{P}=\wt{M}\ltimes N$ be the parabolic subgroup of $\wt{G}_{n^+}$ stabilizing a totally isotropic space flag 
\[
X_{m_1}\subset \cdots \subset X_{m_t},
\]
and we denote by $\wt{P}^-=\wt{M}\ltimes N^-$ the opposite parabolic subgroup of $\wt{P}$, which stabilizes a totally isotropic flag
\[
Y_{m_1}\subset \cdots \subset Y_{m_t}.
\]

For $\pi_{n^+}\in \Pi_F(\wt{G}_{n^+})$, we denote by $\Jac_{\wt{P}}(\pi_{n^+})$ the Jacquet module of $\pi_{n^+}$ with respect to $\wt{P}$ and by $p_{N}$ the projection from $\pi_{n^+}$ to $\Jac_{\wt{P}}(\pi_{n^+})$.

\begin{lem}\label{lem: Casselman's}
With the notations above, the following hold. 
\begin{enumerate}
    \item 
The coinvariant 
\[
(\omega_{n^+,\psi_F})_{{N}\ltimes X_{m_t}}=V_{\omega_{n^+,\psi_F}}/\Span\{\omega_{n^+,\psi_F}(g)v-v\mid g\in N(F)\ltimes X_{m_t}(F),v\in V_{\omega_{n^+,\psi_F}}\}
\]  
is equal to $\omega_{n^+-m_t,\psi_F}$ is  $\wt{M}(F)$-representation such that  $\GL_{m_{i+1}-m_i}(E)$ acts trivially on it and $G_{n^+-m_t}(F)$ acts as Heisenberg-oscillator representation.
\item There is a nondegenerate $\wt{M}(F)$-equivariant pairing
\[
(\omega_{n^+,\psi_F})_{{N}\ltimes X_{m_t}}\times (\omega_{n^+,\psi_F^{-1}})_{N^-\ltimes Y_{m_t}}\to \BC.
\]
\end{enumerate}
\end{lem}
\begin{proof}
From \cite[Lemma 3.55]{GKT}, we have
\[
(\omega_{n^+,\psi_F})_{X_{m_t}}=\omega_{n^+-m_t,\psi_F}.
\]
With the mixed model of the Weil representation $\omega_{n^+,\psi_F}$ that is realized on $\mathcal{S}(X_{m_t}) \otimes \omega_{n^+-m_t,\psi_F}$, it is clear that ${N}(F)$ acts trivially on $\omega_{n^+-m_t,\psi_F}$. This proves Part (1).

Part (2) follows from the nondegenerate $\wt{G}_{n^+-m_t}(F)$-equivariant bilinear pairing
\[
\omega_{n^+-m_t,\psi_F}\times \omega_{n^+-m_t,\psi_F^{-1}}\to \BC.
\]
\end{proof}

Now, we are ready to do the explicit computation by using the parameterization
\begin{equation}\label{equ: Cartan}
A^+_{G_{n^+}}=\set{{\varpi}^{\udl{n}}=\diag(\varpi^{n_1}_E,\cdots,\varpi_E^{n_r}) | \udl{n}\in {\CN}^+},\text{ where } {\CN}^+=\set{(n_1,\cdots,n_r) | n_1\geq\cdots\geq n_{r}\geq 0 }.
\end{equation}

 For fixed $v,v^*,v',w,v^{\prime,*},w^*$, the smoothness of the representations implies the existence of an open compact subgroup $K_0$ of $K_{\wt{G}_{n^+}}$ such that 
\[
\pi_{K}(K_0)v=v,\quad \sigma_{K}(K_0)v'=v',\quad \omega_{n^+,\psi_F}(K_0,0)w=w,
\]
\[
\pi_K^*(K_0)v^*=v^*,\quad \sigma_{K}(K_0)v^{\prime,*}=v^{\prime,*},\quad \omega_{n^+,\psi_F^{-1}}(K_0,0)w^*=w^*.
\]
From the Cartan decomposition,
\[
\wt{G}_{n^+}(F)=\bigcup_{\udl{n}\in \CN^+}K_{\wt{G}_{n^+}}\varpi^{\udl{n}} K_{\wt{G}_{n^+}}.
\]
Then we express the integral as the infinite sum
\begin{equation}\label{equ: infinite sum}\RI_{v,v^*,v'\otimes w,v^{\prime,*}\otimes w^*}(\udl{s},\udl{s}')=\frac{1}{[K_{\wt{G}_n}:K_0]^2}\sum_{\udl{n}\in \CN^+}\sum_{x,x^*\in K_{\wt{G}_n}/K_0
}
F_{\udl{n},x,x'}(\udl{s},\udl{s'}),
\end{equation}
where
\[
\begin{aligned}
F_{\udl{n},x,x'}(\udl{s},\udl{s'})
&= \mathrm{meas}(K_{\wt{G}_{n^+}}{\varpi}^{\udl{n}}K_{\wt{G}_{n^+}})\langle\pi_{\udl{s}}(\wt{{\varpi}^{\udl{n}}})\pi_K(x)v, \pi_K^*(x^*)v^*\rangle \langle\sigma_{\udl{s'}}(\wt{{\varpi}^{\udl{n}}})\sigma_K(x)v',\sigma_K^*(x^*)v^{\prime,*}\rangle\\
&\qquad\cdot\langle \omega_{n^+,\psi_F}(\wt{{\varpi}^{\udl{n}}},0)\omega_{n^+,\psi_F}(x)w, \omega_{n^+,\psi_F^{-1}}(x^*,0)w^*\rangle.
\end{aligned}
\]
From Casselman's canonical pairing for $\pi_{\udl{s}}$, $\sigma_{\udl{s}'}$ and Lemma \ref{lem: Casselman's}, we have that 
\[
\begin{aligned}
F_{\udl{n},x,x'}(\udl{s},\udl{s'})
&=\mathrm{meas}(K_{\wt{G}_{n^+}}{\varpi}^{\udl{n}}K_{\wt{G}_{n^+}})
\langle\pi_{\udl{s}}(\wt{{\varpi}^{\udl{n}}})p_{\wt{N}}(\pi_K(x)v),p^*_{\wt{N}^-}( \pi_K^*(x^*)v^*)\rangle \\
&\qquad\cdot\langle\sigma_{\udl{s'}}(\wt{{\varpi}^{\udl{n}}})p_{\wt{N}}(\sigma_K(x)v'),p^*_{\wt{N}^-}(\sigma_K^*(x^*)v^{\prime,*})\rangle\\
&\qquad\qquad \cdot\langle \omega_{n^+,\psi_F}(\wt{{\varpi}^{\udl{n}}},0)p_{X_{m_t}}(\omega_{n^+,\psi_F}(x,0)w), \omega_{n^+,\psi_F^{-1}}(x^*,0)p^*_{Y_{m_t}}(w^*)\rangle.
\end{aligned}
\]
For fixed $x,x'\in K_{\wt{G}_{n^+}}/K_0$, we have 
\[
F_{\udl{n},x,x'}(\udl{s},\udl{s'})=C_{x,x'}\mathrm{meas}(K_{\wt{G}_{n^+}}{\varpi}^{\udl{n}}K_{\wt{G}_{n^+}})\chi_{\udl{s}}(\varpi^{\udl{n}})\chi_{\udl{s'}}(\varpi^{\udl{n}}),
\]
where $\chi_{\udl{s}}$ and $\chi_{\udl{s}'}$ are the central characters of the representations 
\[
\delta_{P}^{1/2}|\det|^{s_1}\rho_1 \otimes \cdots \otimes |\det|^{s_t}\rho_t \otimes \pi_0
\quad\text{and}\quad
\delta_{P'}^{1/2}|\det|^{s_1'}\tau_1 \otimes \cdots \otimes |\det|^{s_{t'}'}\tau_t \otimes \sigma_0
\]respectively.
From \cite[Proposition (3.2.15)]{M71}, we have
\[
\mathrm{meas}(K_{\wt{G}_{n^+}}\varpi^{\udl{n}}K_{\wt{G}_{n^+}})=\delta_{P_0}(\varpi^{\udl{n}})Q^{\udl{n}}(q^{-1}),
\]
where $Q^{\udl{n}}$ only depends on the set $S_{\udl{n}}=\set{1\leq i\leq r | n_i=n_{i+1}}$, here we take $n_{r+1}=0$.
%Finally, we conclude Proposition \ref{pro: integral}(3) by explicit computation.
%\[
%\begin{aligned}
%&\sum_{\udl{n}}\mathrm{meas}(K{\varpi}^{\udl{n}}K)\chi_{\udl{s}}(\varpi^{\udl{n}})\chi_{\udl{s'}}(\varpi^{\udl{n}})    \\
%=&\delta_{P_0}({\varpi}^{\udl{n}})%\delta_P^{1/2}%({\varpi}^{\udl{n}})\delta_P^{1/2}({\varpi}^{\udl{n}})
%q^{-\langle\udl{s},\udl{n}\rangle-\langle\udl{s'},\udl{n}\rangle}
%\end{aligned}
%\]
When 
\[
|\Re(s_i)|,\ |\Re(s_i')|<\frac{1}{4r}
\]
for every $1\leq i\leq r$, the infinite sum
\[
\sum_{\substack{\udl{n}\in\CN^+\\S_{\udl{n}}=S}}
F_{\udl{n},x,x'}(\udl{s},\udl{s'})=\sum_{\substack{\udl{n}\in\CN^+\\S_{\udl{n}}=S}}C_{x,x'}
\delta_{P_0}({\varpi}^{\udl{n}})\chi_{\udl{s}}(\varpi^{\udl{n}})\chi_{\udl{s'}}(\varpi^{\udl{n}})
\]
converges to a rational function.   From (\ref{equ: infinite sum}), we conclude Part (3) of Proposition \ref{pro: integral}.
%\appendix

\end{document}